\documentclass[onecolumn]{autart}    % Enable this line and disable the 
\usepackage{graphicx}          % Include this line if your 
\usepackage{epstopdf} 
\usepackage[english]{babel} %ESPAOL
\usepackage[utf8]{inputenc}
\usepackage{multicol}%PARA ESCRIBIR COSAS EN DOS COLUMNAS
\usepackage{amsmath} %COSAS DE MATEMATICAS
\usepackage{enumerate} %PARA HACER LAS VIETAS
\usepackage{booktabs} %MODIFICACIN DE TABLAS
\usepackage{array}
\usepackage{rotating} %rotacion de texto
\usepackage{multirow} %combinacion de filas en tablas
\usepackage{float}
\usepackage{color}
\usepackage{lineno}

\newtheorem{theorem}{Theorem}[section]
\newtheorem{corollary}[theorem]{Corollary}
\newtheorem{lemma}[theorem]{Lemma}

\theoremstyle{definition}
\newtheorem{definition}[theorem]{Definition}
\newtheorem{remark}[theorem]{Remark}
\newenvironment{proof}{\noindent{\it Proof.}}

\begin{document}

\begin{frontmatter}
%\runtitle{Insert a suggested running title}  % Running title for regular 
                                              % papers but only if the title  
                                              % is over 5 words. Running title 
                                              % is not shown in output.

%\title{CONTROLABILIDADE SIMULTANEA DE UM PROBLEMA DE TRANSMISSÃO NA PROPAGAÇÃO DE ONDAS DE SOM\\}
\title{ Simultaneous controllability of wave sound propagations}%\thanksref{footnoteinfo}} % Title, preferably not more 
                                                % than 10 words.

%\thanks[footnoteinfo]{This paper was not presented at any IFAC 
%meeting. Corresponding author M.~T.~Cicero. Tel. +XXXIX-VI-mmmxxi. 
%Fax +XXXIX-VI-mmmxxv.}

\author[Paestum]{Alexis Rodriguez Carranza}\ead{alexis@pg.im.ufrj.br},    % Add the 
\author[Rome]{Luis J. Caucha}\ead{ljhony82@yahoo.com}
\author[Paestum]{Obidio Rubio Mercedes}\ead{obidior@yahoo.co.uk}
%\author[Baiae]{Marco A. P. Cabral}\ead{mapcabral@ufrj.br}  % (ead) as shown
%\author[Rome]{Juan Ponte Bejarano}\ead{juan@upn.edu.pe}.               % e-mail address 
\address[Paestum]{Universidad Nacional de Trujillo, Perú}  % Please supply    
\address[Rome]{Universidad Nacional de Tumbes, Perú}                                           
%\address[Baiae]{Universidade Federal de Rio de Janeiro, Brasil}        
%\address[Rome]{Universidad Privada del Norte, Perú}             % full addresses
% here.

                             % keyword list or with the 
                                          % help of the Automatica 
                                          % keyword wizard

\begin{abstract}              % Abstract of not more than 200 words.
In this work we study one problem of mathematical interest for their applications in several topics in Applied Science. We study simultaneous controllability of a pair of systems which model the evolution of sound in a compressible flow considered as a transmission problem. We show the well posed of the problem. Furthermore provided appropriate conditions in the geometry of the domain are valid and suitable assumptions on the fluid, is possible to conduce the pair of systems to the equilibrium in a simultaneous way using only one control.
\end{abstract}

\end{frontmatter}
%\linenumbers

\section{Introduction}
%Neste trabalho considerarmos um par de sistemas de equações que descrevem a evolução do som em fluidos compressíveis. Um modelo linear bem conhecido, veja \cite{Leis}, é dado pelo sistema 
In this work, we considered an equations system to describe an evolution of the wave sound or compressible fluids. A linear model well know is given by a system \cite{Leis}
\begin{equation}{\label{Par_Ecu_Ondas_ST1}}
		\left\{  \begin{aligned}
		& \frac{\partial u}{\partial t}+\alpha \nabla p = 0, 				  \quad \mbox{in}\quad\Omega\times(0,T) \\
		& \frac{\partial p}{\partial t}+\beta\mbox{div}(u) = 0,                  	       \quad \mbox{in}\quad\Omega\times(0,T) \\		
		& u.\eta  = Q, \quad \mbox{in}\quad S_0\times(0,T)\\
		& p = 0, \quad \mbox{in}\quad S_1\times(0,T)\\
		& u(x,0) = u_{0}(x), p(x,0) =p_{0}(x)
 		\end{aligned}\right.	
\end{equation}

%Onde  $p=p(x,t)$ denota a pressão acústica, $u=(u_1,u_2,u_3)$ com  $u_j=u_j(x,t)$ é o campo de velocidade do fluido, $\alpha>0$ é a densidade de equilíbrio e $\beta>0$ é a compressibilidade do fluido. Aqui $\Omega$ é um aberto limitado do $I\!\!R^3$
%como fronteira regular $S_0\cup S_1=\partial\Omega$ e $S_0\cap S_1=\emptyset$

%Para tratar o problema de controlabilidade simultânea vamos considerar também o seguinte sistema.\\

 where  $p=p(x,t)$ is acoustic precision, $u=(u_1,u_2,u_3)$ and  $u_j=u_j(x,t)$ are fluid velocity field, $\alpha>0$ is the density of equilibrium and  $\beta>0$ is the compressibility factor of fluid. Here $\Omega$ is an open subset of $I\!\!R^3$ with regularity boundary conditions $S_0\cup S_1=\partial\Omega$ and $S_0\cap S_1=\emptyset$.\\
To solve the simultaneous controllability we considered a system given by

\begin{equation}{\label{Par_Ecu_Ondas_ST2}}
		 \left\{  \begin{aligned}
         &\frac{\partial v}{\partial t}+\gamma \nabla q = 0, 				\quad \mbox{in}\quad\Omega\times(0,T) \\
		& \frac{\partial q}{\partial t}+\tau\mbox{div}(v) = 0, 				\quad \mbox{in}\quad\Omega\times(0,T) \\		
		& q  = P, \quad \mbox{in}\quad S_0\times(0,T)\\
		& q = 0, \quad \mbox{in}\quad S_1\times(0,T)\\
		& v(x,0) =v_{0}(x), q(x,0) =q_{0}(x)
		\end{aligned}\right.
\end{equation}

%onde $\gamma>0$ e $\tau>0$. As funções $Q$ e $P$ em (\ref{Par_Ecu_Ondas_ST1}) e (\ref{Par_Ecu_Ondas_ST2}), respectivamente, são chamadas de funções de controle. 

%Em torno do $1986$, D.L. Russell\cite{Russell_2} e J.L.Lions \cite{Lions} perguntaram se for possível resolver o problema de controlabilidade exata para um par de modelos de evolução usando somente uma única função de controle. Eles chamam esse problema como um problema de controlabilidade simultânea. Na ausência de efeitos dissipativos, como no caso considerado em (\ref{Par_Ecu_Ondas_ST1}) e (\ref{Par_Ecu_Ondas_ST2}), o problema presenta dificuldades técnicas a superar, veja por exemplo \cite{Lagnese_3}, \cite{Kapitonov}, \cite{Perla} e \cite{Lions}, onde eles perturbam adequadamente os multiplicadores que irão a usar para obter controlabilidade.

%Concretamente o problema de controlabilidade simultânea para os sistemas (\ref{Par_Ecu_Ondas_ST1}) e (\ref{Par_Ecu_Ondas_ST2}) consiste em controlar ambos sistemas usando uma única função de controle, i.e., dado um $T>0$ e quaisquer dado inicial, $(u_0,p_0,v_0,q_0)$,  e final $(\tilde{u}_0,\tilde{p}_0,\tilde{v}_0,\tilde{q}_0)$ em espaços funcionais adequados, achar $P(x,t)$ e $Q(x,t)$ tais que 

where $\gamma>0$ and $\tau>0$.  $Q$ and $P$ in  (\ref{Par_Ecu_Ondas_ST1}) and (\ref{Par_Ecu_Ondas_ST2}), respectively; these are  control functions.
In $1986$, D.L. Russell\cite{Russell_2} and J.L.Lions \cite{Lions} proposed to solve a exact controllability problem for an  evolution model, using only one control function. They called that problem as simultaneous controllability.  The absences of dissipative effects as in (\ref{Par_Ecu_Ondas_ST1}) and (\ref{Par_Ecu_Ondas_ST2}), the problem present difficulties for the solution, see the examples \cite{Lagnese_3}, \cite{Kapitonov}, \cite{Perla} and \cite{Lions}, where they perturbed the multipliers used for the controllability.\\
The problem of simultaneous controllability for the systems (\ref{Par_Ecu_Ondas_ST1}) and (\ref{Par_Ecu_Ondas_ST2}) is to take a control for both of system using only one control function, i.e., given  $T>0$ any initial condition, $(u_0,p_0,v_0,q_0)$,  and final $(\tilde{u}_0,\tilde{p}_0,\tilde{v}_0,\tilde{q}_0)$ in appropriate functional space, find  $P(x,t)$ and  $Q(x,t)$ such that

%\begin{enumerate}
%	\item[a)] A solução $\{u,p,v,q \}$ de 										(\ref{Par_Ecu_Ondas_ST1}) e 										(\ref{Par_Ecu_Ondas_ST2}) satisfazem no tempo $T$
%			\[(u(.,T),p(.,T),v(.,T),q(.,T))=(\tilde{u}_0,						\tilde{p}_0,\tilde{v}_0,\tilde{q}_0)\]
%	\item[b)] A função de controle, $P(x,t)$, para 						(\ref{Par_Ecu_Ondas_ST2}) seja dada em termos de $Q(x,t)$, ou ao contrario.
%\end{enumerate}
%\textcolor{red}{
\begin{enumerate}
   \item[a)] A solution  $\{u,p,v,q \}$ of (\ref{Par_Ecu_Ondas_ST1}) and (\ref{Par_Ecu_Ondas_ST2}) satisfied in $T$
			\[(u(.,T),p(.,T),v(.,T),q(.,T))=(\tilde{u}_0, \tilde{p}_0,\tilde{v}_0,\tilde{q}_0)\]
	\item[b)] The control function, $P(x,t)$, for 	(\ref{Par_Ecu_Ondas_ST2}) was given in terms of $Q(x,t)$.

\end{enumerate}
%}

%Um dos métodos para resolver problemas de controlabilidade é o Hilbert Uniqueness Method (H.U.M) introduzido por J.L.Lions ele é baseado na construção de apropriadas estruturas de espaços de Hilbert no espaço de dados iniciais. Estas estruturas estão conectadas com propriedades de unicidade. Importantes contribuições sobre a controlabilidade para os problemas (\ref{Par_Ecu_Ondas_ST1}) e (\ref{Par_Ecu_Ondas_ST2}) foram feitos por Kapitonov e G. Perla Menzala \cite{Kapitonov}, \cite{Perla}. Em \cite{Perla} e \cite{Kapitonov} os autores responderam de maneira afirmativa para o controle simultâneo e o mostram que o controle $P=-\frac{\beta}{\gamma} Q$ pode ser usado para resolver o problema. 

%Neste trabalho estudamos o problema de controlabilidade desses sistemas desde um ponto de vista mas interessante para as aplicações, a souber, o chamado problema de transmissão o qual vamos a descrever a continuação.

A method to solve the controllability problem is Hilbert Uniqueness Method (H.U.M) proposed by J.L.Lions, it is a construction  of an appropriate structure for the Hilbert space in the initial conditions space.\\
These structure are connected by uniqueness properties.  An important contribution to the controllability problems (\ref{Par_Ecu_Ondas_ST1}) and (\ref{Par_Ecu_Ondas_ST2}) were made by Kapitonov et. G. Perla Menzala \cite{Kapitonov}, \cite{Perla}. In  \cite{Perla} and \cite{Kapitonov} the author answered positively for a simultaneous control and They showed that the control  $P=-\frac{\beta}{\gamma} Q$ could be use to solve a problem.
In this work we study a controllability problem of these systems with a perspective for applications as a problem of transmission; this is described below.

%sejam $\sigma_0$ e $\sigma_1$ abertos limitados e conexos de $i\!\!r^{3}$, com $\bar{\sigma}_1\subseteq \sigma_0$. seja $\omega=\sigma_0\setminus \bar{\sigma}_1$ e denotemos por $\partial\sigma_{0}=s_{0}, \quad\partial \sigma_{1}=s_1$. fixemos um inteiro $m>1$ e seja $k=1,2,\dots,m$. para cada $k$, seja $b_{k}$ um subconjunto aberto e conexo, com fronteira regular e tal que, $\bar{\sigma}_1\subseteq b_k\subseteq\sigma_0$, $\bar{b}_k\subseteq b_{k+1}$. ponha $\omega_{0}=b_{1}\setminus\bar{\sigma}_1$, $\omega_k = b_{k+1}\setminus\bar{b}_k$, $k=1,2,\dots,m-1$ e $\omega_{m}=\sigma_{0}\setminus\bar{b}_m$. 

%assim, $\omega=\cup^{m}_{j=0}\omega_j$, para $i\neq j$, tem-se $\omega_i \cap \omega_j =\emptyset$ e $\partial\omega= s_0 \cup s_1$. exemplos dessa decomposição é mostrada nas seguintes figura

	Given  $\sigma_0$ and $\sigma_1$ open limited subset and conexo in  $i\!\!r^{3}$, with  $\bar{\sigma}_1\subseteq \sigma_0$. Also $\omega=\sigma_0\setminus \bar{\sigma}_1$, we denoted $\partial\sigma_{0}=s_{0}, \quad\partial \sigma_{1}=s_1$. And fixed an integer $m>1$ and  $k=1,2,\dots,m$. For each $k$,  $b_{k}$ is an open subset and conexo, with regularity in the boundary such that, $\bar{\sigma}_1\subseteq b_k\subseteq\sigma_0$, $\bar{b}_k\subseteq b_{k+1}$. We put  $\omega_{0}=b_{1}\setminus\bar{\sigma}_1$, $\omega_k = b_{k+1}\setminus\bar{b}_k$, $k=1,2,\dots,m-1$ and $\omega_{m}=\sigma_{0}\setminus\bar{b}_m$. 
and, $\omega=\cup^{m}_{j=0}\omega_j$, for  $i\neq j$, we take  $\omega_i \cap \omega_j =\emptyset$ and $\partial\omega= s_0 \cup s_1$. Examples for this decomposition is showed in

\begin{figure}[H]
	\begin{center}
		\includegraphics[scale=0.7]{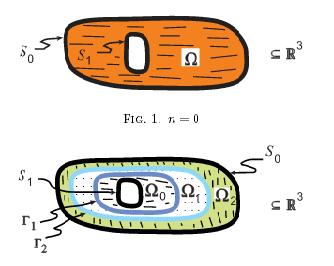}
		\caption{Case $m=0$ and $m=3$}
		\label{fig-multilayer}
	\end{center}
\end{figure}

%Procura-se por uma solução definida por partes, em cada sub domínio, para isso, considere os sistemas (\ref{Par_Ecu_Ondas_ST1}) e (\ref{Par_Ecu_Ondas_ST2}) restritos aos sub domínios $\Omega_k$, assim
 We need a solution defined by part on each sub domain; for that, we considered the systems (\ref{Par_Ecu_Ondas_ST1}) and (\ref{Par_Ecu_Ondas_ST2}) rewrite on sub domains $\Omega_k$,  and

\begin{equation}{\label{Par_Ecu_Ondas_Multilayer_1}}
	\left\{  \begin{array}{ccc}
				\frac{\partial u^{k}}{\partial t}+\alpha^{k} \nabla p^{k} &=& 0, \quad 							\mbox{in}\quad \Omega_{k}\times(0,T) \\
				\frac{\partial p^{k}}{\partial t}+\beta^{k}\mbox{div}(u^{k}) &=& 0, \quad 						\mbox{in}\quad\Omega_{k}\times(0,T) \\
				u^{k}(x,0) =&u^{k}_{0}(x),& p^{k}(x,0) =p^{k}_{0}(x)
 			 \end{array}\right. 
 			 \end{equation}

\begin{equation}{\label{Par_Ecu_Ondas_Multilayer_2}}
	\left\{  \begin{array}{ccc}
				\frac{\partial v^{k}}{\partial t}+\gamma^{k} \nabla q^{k} &=& 0, \quad 							\mbox{in}\quad\Omega_{k}\times(0,T) \\
				\frac{\partial q^{k}}{\partial t}+\tau^{k}\mbox{div}(v^{k}) &=& 0, \quad 						\mbox{in}\quad\Omega_{k}\times(0,T) \\		
				v^{k}(x,0) =&v^{k}_{0}(x),& q^{k}(x,0) =q^{k}_{0}(x)
 			 \end{array}\right.
 			 \end{equation}

$k=0,1,2,\dots,m$.

%com condições de contorno (\ref{Par_Ecu_Ondas_ST1}) e (\ref{Par_Ecu_Ondas_ST2}). As condições de transmissão nas interfaces $\Gamma_{k}=\partial\Omega_k$, dadas por:

 with boundary conditions (\ref{Par_Ecu_Ondas_ST1}) and (\ref{Par_Ecu_Ondas_ST2}). The interfaces of transmission conditions $\Gamma_{k}=\partial\Omega_k$, given by

\begin{equation}{\label{Cond_Multilayer_1}}
	\left\{  \begin{array}{ccc}
				\alpha^{k-1} p^{k-1} &=& \alpha^{k} p^{k}\\
				\beta^{k-1}(u^{k-1}.\eta) &=& \beta^{k}(u^{k}.\eta)\\
				k=2,\dots, m, & (x,t)\in \Gamma_{k}\times (0,T)
 			 \end{array}\right. 
 			 \end{equation}

\begin{equation}{\label{Cond_Multilayer_2}}
	\left\{  \begin{array}{ccc}
				\gamma^{k-1} q^{k-1} &=& \gamma^{k} q^{k}\\
				\tau^{k-1}(v^{k-1}.\eta) &=& \tau^{k}(v^{k}.\eta)\\
				k=2,\dots, m, & (x,t)\in \Gamma_{k}\times (0,T)
 			 \end{array}\right. 
 			 \end{equation}

%para os sistemas (\ref{Par_Ecu_Ondas_Multilayer_1}) e (\ref{Par_Ecu_Ondas_Multilayer_2}) respectivamente. 

%As funções $\alpha^{k}, \beta^{k}, \gamma^{k}$ e $\tau^{k}$ são as restrições das funções $\alpha,\beta,\gamma,\tau$ que aparecem nas equações (\ref{Par_Ecu_Ondas_ST1}) e (\ref{Par_Ecu_Ondas_ST2}), as quais assumimos que são funções constantes por partes, estritamente positivas e podem perder a continuidade somente em $\Gamma_k$, $k=1,2,\dots,m$.

%O objetivo nesta parte será obter uma estimativa da forma 

for the systems (\ref{Par_Ecu_Ondas_Multilayer_1}) and (\ref{Par_Ecu_Ondas_Multilayer_2}), respectively.\\ 
The functions $\alpha^{k}, \beta^{k}, \gamma^{k}$ and $\tau^{k}$ are the restriction for the functions $\alpha,\beta,\gamma,\tau$ on the systems (\ref{Par_Ecu_Ondas_ST1}) and (\ref{Par_Ecu_Ondas_ST2}), we assumed that those functions were constant by parts, strictly positive and we lost the continuity only in  $\Gamma_k$, $k=1,2,\dots,m$.\\
The objective in this section is to get the estimation of

\begin{equation}{\label{D_O_I}}
	\begin{aligned}
		&(T-T_0)\sum_{k=0}^{m}\int_{\Omega_k}\left[ \beta^k \mid u^k \mid^2 + \alpha^k(p^k)^2 + 		\tau^k \mid v^k \mid^2 + \gamma^k(q^k)^2 \right]dx \\
		& \leq C\int_{0}^{T}\int_{S_0}\left[ \alpha p- \tau( v.\eta) \right]^2\frac{\partial h}{\partial\eta}dS_0 dt.
	\end{aligned}
\end{equation}
%para algum $T_0>0$, $C>0$ e qualquer $T>T_0$. A desigualdade (\ref{D_O_I}) é chamada de uma desigualdade de observabilidade a qual será provado no Teorema \ref{Desigualdade_Final} assumindo propriedades geométricas no domínio $\Omega$ e nas interfaces $\Gamma_k$. Alem disso para provar (\ref{D_O_I}) serão assumidas condições de monotonicidade nos coeficientes dos sistemas (\ref{Par_Ecu_Ondas_Multilayer_1}) e (\ref{Par_Ecu_Ondas_Multilayer_2}). A necessidade dos requerimentos mencionados já tinha sido notado por Lions\cite{Lions} no seu estudo de problemas de transmissão. Lagnese\cite{Lagnese_3} também uso as mesmas hipóteses para provar resultados de controlabilidade para uma ampla classe de problemas hiperbólicos. 

%A seguir o roteiro do que será feito.

For some $T_0>0$, $C>0$ and  $T>T_0$. The inequality  (\ref{D_O_I}) is named from an inequality of observation which is in the theorem \ref{Desigualdade_Final} assuming geometrical properties on domain $\Omega$ and in the interfaces  $\Gamma_k$.
 Such that, to prove (\ref{D_O_I}) we assumed monotonicity conditions in the coefficients of the systems (\ref{Par_Ecu_Ondas_Multilayer_1}) and (\ref{Par_Ecu_Ondas_Multilayer_2}). The requirement necessary were found by Lions\cite{Lions} in his study of transmission problem. Lagnese\cite{Lagnese_3} used the same hypothesis to prove the result of controllability for a hyperbolic problem.\\ 
Then this will be done.

\begin{enumerate}
 	\item We showed that  (\ref{Par_Ecu_Ondas_Multilayer_1})-(\ref{Cond_Multilayer_2}) is well posed problem, we used the semigroup theory \cite{Pazy}
 	\item We obtained an inequality of simultaneous observability, for both systems  (\ref{Par_Ecu_Ondas_Multilayer_1}) and (\ref{Par_Ecu_Ondas_Multilayer_2}), these we solved using a multipliers technique\cite{Komornik}
 	\item We applied the  H.U.M (Hilbert Uniqueness Method) to obtain the simultaneous controllability \cite{Lions}.
\end{enumerate}

\section{Functional spaces}

%Considere o espaço de Hilbert $X_1=\left[ L^{2}(\Omega)\right]^{3}\times\left[ L^{2}(\Omega)\right]$, associado a (\ref{Par_Ecu_Ondas_Multilayer_1}). Defina-se um produto interno em $X_1$, dado da seguinte forma, se $(\tilde{u},\tilde{p}),(u,p)\in X_1$, então:

Given the Hilbert space $X_1=\left[ L^{2}(\Omega)\right]^{3}\times\left[ L^{2}(\Omega)\right]$, associate to  (\ref{Par_Ecu_Ondas_Multilayer_1}). We define an scalar product in  $X_1$, given by  $(\tilde{u},\tilde{p}),(u,p)\in X_1$, then:

\begin{equation}{\label{Prod_Int_1}}
\left< (u,p), (\tilde{u},\tilde{p})  \right>_{X_1}=\sum^{m}_{k=0}\int_{\Omega_k}\left\{ \beta^{k} u^{k} \Large\textbf{.}\tilde{u}^{k}+\alpha^{k}p^{k}\tilde{p}^{k} \right\}dx
\end{equation}
%Analogamente, considere $X_2=\left[ L^{2}(\Omega)\right]^{3}\times\left[ L^{2}(\Omega)\right]$ associado a (\ref{Par_Ecu_Ondas_Multilayer_2}). Defina-se um produto interno em $X_2$, como segue, dados $(\tilde{v},\tilde{q}),(v,q)\in X_2$, então:

Consequently, we considered $X_2=\left[ L^{2}(\Omega)\right]^{3}\times\left[ L^{2}(\Omega)\right]$ associate to (\ref{Par_Ecu_Ondas_Multilayer_2}). We define a scalar product in $X_2$, as  $(\tilde{v},\tilde{q}),(v,q)\in X_2$, then:

\begin{equation}{\label{Prod_Int_2}}
\left< (v,q), (\tilde{v},\tilde{q})  \right>_{X_2}=\sum^{m}_{k=0}\int_{\Omega_k}\left\{ \tau^{k} v^{k}\Large\textbf{.}\tilde{v}^{k}+\gamma^{k}q^{k}\tilde{q}^{k} \right\}dx
\end{equation}

%Com as considerações feitas acima considere a energia total associada ao problema (\ref{Par_Ecu_Ondas_Multilayer_1}), (\ref{Par_Ecu_Ondas_Multilayer_2}), (\ref{Cond_Multilayer_1}), (\ref{Cond_Multilayer_2}) e as condições na fronteira dadas em (\ref{Par_Ecu_Ondas_ST1}), (\ref{Par_Ecu_Ondas_ST2}), dado por

We have considered a total energy to the problem (\ref{Par_Ecu_Ondas_Multilayer_1}), (\ref{Par_Ecu_Ondas_Multilayer_2}), (\ref{Cond_Multilayer_1}), (\ref{Cond_Multilayer_2}) and the boundary conditions in (\ref{Par_Ecu_Ondas_ST1}), (\ref{Par_Ecu_Ondas_ST2}), as

\begin{equation}{\label{Energia_Total_P1}}
E(t)=\frac{1}{2}\sum_{k=0}^{m}\int_{\Omega_k}\left\{   \beta_k\mid u^k \mid^2 + \alpha^k (p^k)^{2} + \tau_k\mid v^k \mid^2 + \gamma^k (q^k)^{2} \right\}dx
\end{equation}

%A seguir um Lema útil que permite dar um sentido rigoroso para as condições dadas nas interfaces, veja\cite{Perla} para mais detalhes.

Making a rigorously way for the interfaces conditions, we can see a lemma ~\ref{Lema_Aux_Trans}; for more details see Perla et al.\cite{Perla}.

%\begin{lemma}{\label{Lema_Aux_Trans}}
%Seja $\Omega$ uma região limitada no $I\!\!R^{3}$, com fronteira regular $\partial\Omega$. A aplicação 
%\[
%\begin{array}{ccc}
%\left[ C^{1}(\bar{\Omega})  \right]^{3} &\rightarrow & C^{1}(\partial\Omega)\\
%u=(u_1,u_2,u_3) & \rightarrow & u\Large\textbf{.}\eta
%\end{array}
%\]
%onde $\eta=\eta(x)$ é o vetor normal unitário exterior em $x\in\partial\Omega$. Podemos estender por continuidade a uma aplicação
%\[
%\tilde{H}\longrightarrow H^{-1/2}(\partial\Omega)
%\]
%onde $\tilde{H}=\left\{ u\in \left[ L^{2}(\Omega) \right]^{3},\quad \mbox{tal que},\quad \mbox{div}(u)\in L^{2}(\Omega) \right\}$ e $H^{-1/2}(\partial\Omega)$ é o espaço dual de $H^{1/2}(\partial\Omega)$
%\end{lemma}

\begin{lemma}{\label{Lema_Aux_Trans}}
Given $\Omega$ bounded region in $I\!\!R^{3}$, with regularity in the boundary $\partial\Omega$.  The application  
\[
\begin{array}{ccc}
\left[ C^{1}(\bar{\Omega})  \right]^{3} &\rightarrow & C^{1}(\partial\Omega)\\
u=(u_1,u_2,u_3) & \rightarrow & u\Large\textbf{.}\eta
\end{array}
\]
where $\eta=\eta(x)$ is as exterior unit normal vector in $x\in\partial\Omega$. We can extend by continuity application
\[
\tilde{H}\longrightarrow H^{-1/2}(\partial\Omega)
\]
where $\tilde{H}=\left\{ u\in \left[ L^{2}(\Omega) \right]^{3},\quad \mbox{such that},\quad \mbox{div}(u)\in L^{2}(\Omega) \right\}$ and  $H^{-1/2}(\partial\Omega)$ is dual space of  $H^{1/2}(\partial\Omega)$
\end{lemma}

%Para simplificar notações escrevemos $u$ no lugar de $u^k$, $\beta$ no lugar de $\beta^k$, etc., na região $\Omega_k$.

%Pelo lemma \ref{Lema_Aux_Trans} é claro que nos espaços seguintes
%\[
%	\begin{array}{ccc}
%		H_1 = \left\{ (u,p)\in X_1,\quad\mbox{tais que},\quad \left( -\alpha\nabla p, -\beta\mbox{div}(u) \right)\in X_1 \right\} & \subseteq & X_1 \\
%		H_2 = \left\{ (v,q)\in X_2,\quad\mbox{tais que},\quad \left( -\gamma\nabla q, -\tau\mbox{div}(v) \right)\in X_2 \right\} & \subseteq & X_2		 
%	\end{array}
%\]
%podem-se definir os seguintes sub espaços:

To simplify the notation we write  $u^k$ as $u$,  $\beta^k$ as $\beta$, the same way for all symbols in the region $\Omega_k$.
by the lemma \ref{Lema_Aux_Trans} is clearly that the spaces 
\[
	\begin{array}{ccc}
		H_1 = \left\{ (u,p)\in X_1,\quad\mbox{such that },\quad \left( -\alpha\nabla p, -\beta\mbox{div}(u) \right)\in X_1 \right\} & \subseteq & X_1 \\
		H_2 = \left\{ (v,q)\in X_2,\quad\mbox{such that},\quad \left( -\gamma\nabla q, -\tau\mbox{div}(v) \right)\in X_2 \right\} & \subseteq & X_2		 
	\end{array}
\]
we can define the sub spaces:

\[
Z_1 = \left\{ (u,p)\in H_1,\quad \mbox{such that } \begin{array}{cc}
\alpha^{k-1}p^{k-1}=& \alpha^{k}p^{k}\\
\beta^{k-1}(u^{k-1}\Large\textbf{.}\eta)= & \beta^{k}(u^{k}\Large\textbf{.}\eta)\\
u\Large\textbf{.}\eta=0\in S_0,& \quad p=0\in S_1
\end{array}, \quad \mbox{in}\quad \Gamma_{k}, k=2,\dots,m. \right\}
\]
and,
\[
Z_2 = \left\{ (v,q)\in H_2, \mbox{such that} \begin{array}{cc}
\gamma^{k-1}q^{k-1}=& \gamma^{k}q^{k}\\
\tau^{k-1}(v^{k-1}\Large\textbf{.}\eta)= & \tau^{k}(v^{k}\Large\textbf{.}\eta)\\
q=0, &\mbox{in}\quad \partial\Omega=S_0 \cup S_1
\end{array}, \quad\mbox{in}\quad \Gamma_{k},\quad k=2,\dots,m \right\}
\]

%Observe que $\left[ C^{1}(\bar{\Omega}) \right]\subset Z_j, \quad j=1,2$. Assim $Z_1$ e $Z_2$ são densos em $X_1$ e $X_2$ respectivamente. Consideremos os seguintes operadores não limitados
%\[A_j: Z_j=D(A_j)\subseteq X_j\longrightarrow X_j\]
%definidos como
%\begin{enumerate}
%	\item Se $(u,p)\in D(A_1)$, então, $A_{1}(u,p)=\left( -\alpha\nabla p, -\beta\mbox{div}(u)
%	 \right)$
%	\item Se $(v,q)\in D(A_2)$, então, $A_{2}(v,q)=\left( -\gamma\nabla q, -\tau\mbox{div}(v)
%	 \right)$	 
%\end{enumerate}

Observe that $\left[ C^{1}(\bar{\Omega}) \right]\subset Z_j, \quad j=1,2$. Also $Z_1$ and $Z_2$ are dense in  $X_1$ and  $X_2$, respectively. Considering the bounded operator
\[A_j: Z_j=D(A_j)\subseteq X_j\longrightarrow X_j\]
defined as 
\begin{enumerate}
	\item Given $(u,p)\in D(A_1)$, then, $A_{1}(u,p)=\left( -\alpha\nabla p, -\beta\mbox{div}(u)
	 \right)$
	\item Given $(v,q)\in D(A_2)$, then , $A_{2}(v,q)=\left( -\gamma\nabla q, -\tau\mbox{div}(v)
	 \right)$	 
\end{enumerate}

%O operador adjunto de $A_1$, denotado por $A_{1}^{*}$, pode ser calculado e é dado da seguinte forma:

The adjoint operator of $A_1$ is denoted by $A_{1}^{*}$; It is calculated and given as :

\[
A^{*}_{1}(\tilde{u},\tilde{p})=\left( \alpha \nabla\tilde{p},\beta\mbox{div}(\tilde{u}) \right)
\]
and 

\[
D(A_{1}^{*}) = \left\{ (\tilde{u},\tilde{p})\in H_1,\quad\mbox{such that }, \begin{array}{cc}
\alpha^{k-1}\tilde{p}^{k-1}=& \alpha^{k}\tilde{p}^{k}\\
\beta^{k-1}(\tilde{u}^{k-1}.\eta)= & \beta^{k}(\tilde{u}^{k}.\eta)\\
p=0\quad\mbox{in}\quad S_1, & \tilde{u}.\eta=0\quad\mbox{in}\quad S_0
\end{array},\mbox{in}\quad \Gamma_{k},\quad k=2,\dots,m \right\}
\]

%Em \cite{Perla} mostram que o operador $A_1$ é skew-adjoint, i.e, $A_{1}^{*}=-A_1$, o mesmo resultado é provado para $A_2$ . Logo, pelo Teorema de Stone segue-se que $A_1$ e  $A_2$ são geradores infinitesimais de um grupo de operadores unitários fortemente contínuos $\{ U_{j}(t) \}_{t\in I\!\!R}$, em $X_1$ e $X_2$ respectivamente.

%Alem disso, $U_{j}(t)w_j$ é fortemente diferenciável em relação a $t$ e para qualquer $w_j\in D(A_j)$,
%\[
%\frac{d}{dt}U_j (t)w_j = A_{j}U(t)w_j
%\]

%Agora estudamos algumas propriedades das soluções de (\ref{Par_Ecu_Ondas_Multilayer_1}) que serão usadas no que segue.

Perla et al. \cite{Perla} showed that operator $A_1$ is  skew-adjoint, i.e, $A_{1}^{*}=-A_1$, the same result was proved for $A_2$ . Using the Stone's theorem, we have proved that $A_1$ and  $A_2$ generate infinitesimally a group of strongly continuous  unit operators  $\{ U_{j}(t) \}_{t\in I\!\!R}$, in  $X_1$ and $X_2$, respectively.
Moreover, $U_{j}(t)w_j$ is strongly differentiable in relation to $t$ and for any  $w_j\in D(A_j)$,
\[
\frac{d}{dt}U_j (t)w_j = A_{j}U(t)w_j
\]
Now, we study some properties of the solutions (\ref{Par_Ecu_Ondas_Multilayer_1}), these are used in

%\begin{lemma}{\label{Prop_A_j}}
%Seja $V_j=\left[  \mbox{Ker}(A^{*}_{j}) \right]^{\bot}$, considerando a ortogonalidade em relação ao produto interno definido em $X_1$ e $X_2$ respectivamente. Então, são validos:
%	\begin{enumerate}
%		\item $U_{j}(t)\left( V_j\cap D(A_j) \right)\subset V_j \cap D(A_j)$
%		\item Fixe $t\in I\!\!R$. Se $(u,p)\in V_1\cap D(A_1)$, então, no sentido das 					  distribuições
%			\begin{enumerate}
%				\item $\mbox{curl}(u^{k})=0$, em $\Omega_k$, $k=0,1,\dots,m$
%				\item $u\times\eta =0$ em $S_1$
%				\item $u^{k-1}\times\eta =u^{k}\times\eta$ em $\Gamma_k$, $k=2,\dots,m$
%			\end{enumerate}
%			onde $\times$ denota o produto vetorial usual de $I\!\!R^{3}$ e $\eta(x)$ é a 				normal exterior a $\Gamma_k$
%		\item Fixe $t\in I\!\!R$. Se $(v,q)\in V_2\cap D(A_2)$, então, no sentido das 					  distribuições	
 %         \begin{enumerate}
%				\item $\mbox{curl}(v^{k})=0$, em $\Omega_k$, $k=0,1,\dots,m$
%				\item $v\times\eta =0$ em $\Gamma$
%				\item $v^{k-1}\times\eta =v^{k}\times\eta$ em $\Gamma_k$, $k=2,\dots,m$
%			\end{enumerate}
%	\end{enumerate}
%\end{lemma}

\begin{lemma}{\label{Prop_A_j}}
Given $V_j=\left[  \mbox{Ker}(A^{*}_{j}) \right]^{\bot}$, considering the orthogonality in relation with the scalar product defined in $X_1$ and $X_2$, respectively. Then, the following results are valid:
	\begin{enumerate}
		\item $U_{j}(t)\left( V_j\cap D(A_j) \right)\subset V_j \cap D(A_j)$
		\item Fixing $t\in I\!\!R$, and $(u,p)\in V_1\cap D(A_1)$, then, in the distributions manner			  
			\begin{enumerate}
				\item $\mbox{curl}(u^{k})=0$, in $\Omega_k$, $k=0,1,\dots,m$
				\item $u\times\eta =0$ in $S_1$
				\item $u^{k-1}\times\eta =u^{k}\times\eta$ in $\Gamma_k$, $k=2,\dots,m$
			\end{enumerate}
			where $\times$ denoted a vectorial product of  $I\!\!R^{3}$ and $\eta(x)$  is an exterior normal vector $\Gamma_k$
		\item Fixing $t\in I\!\!R$,  $(v,q)\in V_2\cap D(A_2)$, then, in the distribution manner	
          \begin{enumerate}
				\item $\mbox{curl}(v^{k})=0$, in $\Omega_k$, $k=0,1,\dots,m$
				\item $v\times\eta =0$ in $\Gamma$
				\item $v^{k-1}\times\eta =v^{k}\times\eta$ in $\Gamma_k$, $k=2,\dots,m$
			\end{enumerate}
	\end{enumerate}
\end{lemma}

\begin{proof}
The proof of  $A_1$ , $A_2$ are similar. we make the proof for the first one. Given $(u,p)\in V_1\cap D(A_1)$, then $(u,p)\in V_1$ and  $(u,p)\in D(A_1)$. As  consequence of semigroup theory, we know that  $U(t)(u,p)\in D(A_1)$, $\forall t\in I\!\!R$, Now we need to prove that  $U(t)(u,p)\in V_1$. 
See that,  $\mbox{Ker}(A^{*}_{1})\neq \emptyset$ has elements of the form $(\beta^{-1}\mbox{Curl}(v),0)$ where $v\in \left[ H^{2}(\Omega) \right]^{3},\quad v=0$ in $S_0$
Given $w=(w_1,w_2)\in\mbox{Ker}(A_{1}^{*})$, then $A_{1}^{*}(w_1,w_2)=0$. And,
\[
\frac{d}{dt}\left( U_1(t)(u,p),(w_1,w_2) \right)=\left< A_1U_1(t)(v,p),(w_1,w_2) \right>_{X_1} = \left<  U_1(t)(u,p),A^{*}_{1}(w_1,w_2) \right>=0
\]
we have that,
\[U_1(t)(u,p),(w_1,w_2)=\mbox{C}, \quad\mbox{C=constant}\quad\forall t\in I\!\!R.\]
In particular, for $t=0$, $\left<  (u,p),(w_1,w_2) \right>_{X_1}=\mbox{C}$, and , $(u,p)\in \left[  \mbox{Ker}(A^{*}_{1}) \right]^{\perp}$ and $(w_1,w_2)\in \mbox{Ker}(A^{*}_{1})$, this implied that $C=0$.
In the same form, $\left<  U_1(t)(u,p),(w_1,w_2) \right>=0, \quad\forall t\in I\!\!R$. Given $U_{1}(t)(u,p)\in \left[  \mbox{Ker}(A^{*}_{1}) \right]^{\perp}=V_1$ and the item $1.)$ was proved.
Now, we prove the first item $2)$. Given  $v\in \left[ H^{2}(\Omega) \right]^{3}$ with support in $\Omega_k$ and considering the elements $\left(  \beta^{-1}\mbox{Curl}(v),0  \right)\in \mbox{Ker}(A^{*}_{1})$. Then, for all $(u,p)\in V_1\cap D(A_1)$, we have;
\[
0=\left< (u,p), \left(  \beta^{-1}\mbox{Curl}(v),0  \right) \right>_{X_1}=\sum^{m}_{k=0}\int_{\Omega_k}u{\Large\textbf{.}}\mbox{Curl}(v)dx
\]
because the support of $v$ is in  $\Omega_k$. The same form, $\mbox{Curl}(v)=0$ in  $\Omega_k$, $k=0,1,\dots,m$ in the distributions way. 
For proving the  item $b)$ of $2)$ We use the identities such as,
\begin{equation}{\label{Int_Part_Clas}}
\int_{\Omega}\mbox{Curl}(u){\Large\textbf{.}}v dx=\int_{\Omega}u{\Large\textbf{.}}\mbox{Curl}(v) dx - \int_{\partial\Omega}v{\Large\textbf{.}}(u\times\eta)d\Gamma
\end{equation}
Given $v\in \left[ H^{2}(\Omega) \right]^{3}$ and $\left(  \beta^{-1}\mbox{Curl}(v),0  \right)\in\mbox{Ker}(A^{*}_{1})$, as, $v=0$ in  $\bigcup^{j=1}_{m}\bar{\Omega_j}$. Using (\ref{Int_Part_Clas}) we have
\begin{eqnarray}{\label{Int_Par_Aux}}
0 &=& \left< (u,p), \left(  \beta^{-1}\mbox{Curl}(v),0  \right) \right>_{X_1} = \sum^{m}_{k=0}\int_{\Omega_k}u^{k}{\Large\textbf{.}}\mbox{Curl}(v^{k})dx\nonumber \\
&=& \sum^{m}_{k=0}\int_{\Omega_k}\mbox{Curl}(u^{k})\Large\textbf{.}v^{k} dx + \sum^{m}_{k=0}\int_{\partial\Omega_k}v^{k}\Large\textbf{.}(u^{k}\times\eta) d\Gamma_k \\
0 &=& \sum^{m}_{k=0}\int_{\partial\Omega_k}v^{k}\Large\textbf{.}(u^{k}\times\eta) d\Gamma_k = \int_{S_1}v\Large\textbf{.}(u\times\eta)dS_1\nonumber
\end{eqnarray}
The last prove is for the item $c)$ of $2)$. Given  $v\in \left[ H^{2}(\Omega) \right]^{3}$ and $\left(  \beta^{-1}\mbox{Curl}(v),0  \right)\in\mbox{Ker}(A^{*}_{1})$, using the identities (\ref{Int_Par_Aux}), we have
\begin{eqnarray}{\label{Int_Par_Aux_1}}
0 &=& \sum^{m}_{k=0}\int_{\Omega_k}\mbox{Curl}(u^{k}).v^{k} dx + \sum^{m}_{k=0}\int_{\partial\Omega_k}v^{k}\Large\textbf{.}(u^{k}\times\eta) d\Gamma_k \nonumber \\
0 &=& \sum^{m}_{k=0}\int_{\Omega_k}v^{k}\Large\textbf{.}(u^{k}\times\eta) d\Gamma_k
\end{eqnarray}
Now, we chose $v$ such that $v=0$ in  $S_0$ and $v=0$ in $\begin{array}{c}
m\\
\bigcup\\
j=1\\
j\neq k
\end{array}
\Gamma_j$, by (\ref{Int_Par_Aux_1}) we have that
\[0=\int_{\Gamma_k} v\Large\textbf{.}\left\{  u^{k-1}\times\eta-u^{k}\times\eta    \right\}d\Gamma_k \]
that is the expected result.\\
The item $3)$ is proved in the same way.
\end{proof}

The theorem ~\ref{Sol_Fortes_P1} has a summary  of the results .
\begin{theorem}{\label{Sol_Fortes_P1}}
Given $V_j$ an orthogonal complement of the subset $\mbox{Ker}(A^{*}_{j})$, $j=1,2$, in $X_j$. Consider the problems (\ref{Par_Ecu_Ondas_Multilayer_1}), (\ref{Par_Ecu_Ondas_Multilayer_2}), (\ref{Par_Ecu_Ondas_ST1}), (\ref{Par_Ecu_Ondas_ST2}) and the initial conditions  $(u_0,p_0)\in V_1\cap D(A_1)$ and $(v_0,q_0)\in V_2\cap D(A_2)$. Then $(u,p)=U_{1}(t)(u_0,p_0)$ and $(v,q)=U_{2}(t)(v_0,q_0)$ are the uniqueness solutions, respectively. That is,
\[
\begin{array}{cc}
(u,p)\in & C(I\!\!R; V_1\cap D(A_1))\cap C(I\!\!R, X_1)\\
(v,q)\in & C(I\!\!R; V_2\cap D(A_2))\cap C(I\!\!R, X_2)
\end{array}
\]
In addition, these solutions satisfies the properties in the lemma~\ref{Prop_A_j}.
\end{theorem}
Before to show the inequality of observability, we prove some important properties.\\
The energy associate to the systems (\ref{Par_Ecu_Ondas_Multilayer_1}) and  (\ref{Par_Ecu_Ondas_Multilayer_2}), with null boundary conditions, are given by:
\[
E_{1}(t)=\frac{1}{2}\sum_{k=0}^{m}\int_{\Omega_k}\left\{ \beta_k\mid u^k \mid^2 + \alpha^k (p^k)^{2}  \right\} dx
\]
and 
\[
E_{2}(t)=\frac{1}{2}\sum_{k=0}^{m}\int_{\Omega_k}\left\{ \tau_k\mid v^k \mid^2 + \gamma^k (q^k)^{2}  \right\} dx
\]
respectively. To prove that these are not dependent of the time  $t$, In fact, we multiplied the first equation of (\ref{Par_Ecu_Ondas_Multilayer_1}) by  $\beta^k u^k$ and integrating in $\Omega_k$ and adding in $k=0,1,\dots,m$, we have
\begin{equation}{\label{Ener_1}}
\frac{1}{2}\frac{d}{dt}\sum_{k=0}^{m}\int_{\Omega_k}\beta^k\mid u^k \mid^2 dx-\sum_{k=0}^{m}\int_{\Omega_k}\beta^k\alpha^k p^k\mbox{div}(u^k)dx + \sum_{k=0}^{m}\int_{\partial\Omega_k}\beta^k\alpha^k p^k (u^k{\Large\textbf{.}}\eta)dx.
\end{equation}
Multiplying the second equation of (\ref{Par_Ecu_Ondas_Multilayer_1}) by  $\alpha^k p^k$, after integrating in $\Omega_k$ and adding in  $k$, we have that 
\begin{equation}{\label{Ener_2}}
\frac{1}{2}\frac{d}{dt}\sum_{k=0}^{m}\int_{\Omega_k}\alpha^k(p^k)^2 dx + \sum_{k=0}^{m}\int_{\Omega_k}\alpha^k \beta^k p^k\mbox{div}(u^k) dx =0
\end{equation}
adding (\ref{Ener_1}) and (\ref{Ener_2}), we have that 
\[
\frac{1}{2}\frac{d}{dt}\sum_{k=0}^{m}\int_{\Omega_k}\left\{\beta^k\mid u^k \mid^2 +\alpha^k(p^k)^k \right\} dx + \sum_{k=0}^{m}\int_{\partial\Omega_k}\beta^k\alpha^k p^k (u^k\Large\textbf{.}\eta)dx=0
\]
moreover ,
\[  
\begin{aligned}
\sum_{k=0}^{m}\int_{\partial\Omega_k}\beta^k\alpha^k p^k (u^k\Large\textbf{.}\eta)dx &=\int_{S_1}\alpha\beta p (u\Large\textbf{.}\eta)dS_1 + \sum_{k=1}^{m}\int_{\Gamma_k}\left\{  \alpha^{k-1}\beta^{k-1}p^{k-1}(u^{k-1}\Large\textbf{.}\eta)- \right. \\
& \left. \alpha^{k}\beta^{k}p^{k}(u^{k}\Large\textbf{.}\eta)  \right\}d\Gamma_k + 
 \int_{S_0}\alpha\beta p (u.\eta)dS_0
\end{aligned}
\]
Using the contour and interface conditions (\ref{Cond_Multilayer_1}), we have that  $\sum_{k=0}^{m}\int_{\partial\Omega_k}\alpha^k\beta^k p^k (u^k\Large\textbf{.}\eta)d\Gamma_k=0$. Follows the affirmation.  The case $E_2(t)$ is similar .

\section{Inequality of observability}

%Nesta secção obteremos uma desigualdade de observabilidade, a qual será satisfeita para ambos sistemas (\ref{Par_Ecu_Ondas_Multilayer_1}) e (\ref{Par_Ecu_Ondas_Multilayer_2}) simultaneamente. A prova será feita usando a teoria dos multiplicadores, uma boa referencia para o uso desta técnica pode-se encontrar no livro de Komornik\cite{Komornik}.Os multiplicadores usados aqui foram convenientemente modificados com o objetivo de obter boas estimativas dos termos de fronteira. Esses multiplicadores já foram usados por outros autores, motivados pela invariância dos sistemas (\ref{Par_Ecu_Ondas_ST1}) e (\ref{Par_Ecu_Ondas_ST2}), relativos ao grupo de dilatações em todas as varáveis, veja\cite{Kapitonov} e \cite{Perla} e as referencias neles.

%Seja $h:C(\Omega)\cap C^{1}(\bar{\Omega})\longrightarrow I\!\!R$ uma função auxiliar a qual será escolhida depois e $(u,p)\in V_1\cap D(A_1)$ uma solução do sistema (\ref{Par_Ecu_Ondas_ST1}). Consideremos os multiplicadores dados por:

%\[\left\{
%\begin{aligned}
%M_1 &= 2\left( \alpha t p -u\Large\textbf{.}\nabla h + \alpha\int_{0}^{t}p(x,s) ds \right) \\
%M_2 &= 2\left( \beta t u - p\nabla h \right)  \\
%M_3 &= 2\beta u  
%\end{aligned}\right.
%\]

%Sendo $(u,p)$ solução de (\ref{Par_Ecu_Ondas_ST1}), temos a identidade
%\[
%0=M_1\left\{  p_t + \beta\mbox{div}u  \right\}+M_2\Large\textbf{.}\left\{  u_t + \alpha\nabla p  \right\}+ M_3\Large\textbf{.}\left\{  \int_{0}^{t} (u_s +\alpha\nabla p)ds  \right\}
%\]

In this section we show the inequality of observability. This inequality satisfies the systems (\ref{Par_Ecu_Ondas_Multilayer_1}) and (\ref{Par_Ecu_Ondas_Multilayer_2}) simultaneously.  Using the multiplier's theory (see Komornik\cite{Komornik}), we make the proof. The multiplier was modified to get a good  estimates  in the boundary. These multiplier were used in several works. The invariant of the systems (\ref{Par_Ecu_Ondas_ST1}) and  (\ref{Par_Ecu_Ondas_ST2}), in relations to dilatations groups in all variables, see \cite{Kapitonov} and \cite{Perla}.
Given $h:C(\Omega)\cap C^{1}(\bar{\Omega})\longrightarrow I\!\!R$ an auxiliary function, it will be chosen in the next steps; and, given  $(u,p)\in V_1\cap D(A_1)$ a solution of the system (\ref{Par_Ecu_Ondas_ST1}). Considering the multiplier given by:
\[\left\{
\begin{aligned}
M_1 &= 2\left( \alpha t p -u\Large\textbf{.}\nabla h + \alpha\int_{0}^{t}p(x,s) ds \right) \\
M_2 &= 2\left( \beta t u - p\nabla h \right)  \\
M_3 &= 2\beta u  
\end{aligned}\right.
\]
and $(u,p)$ solution of (\ref{Par_Ecu_Ondas_ST1}),  we have the identities
\[
0=M_1\left\{  p_t + \beta\mbox{div}u  \right\}+M_2\Large\textbf{.}\left\{  u_t + \alpha\nabla p  \right\}+ M_3\Large\textbf{.}\left\{  \int_{0}^{t} (u_s +\alpha\nabla p)ds  \right\}
\]

%A expressão acima, pode ser restrita na forma seguinte
%\begin{equation}{\label{Eq_Control}}
%0=\frac{\partial A}{\partial t}-\mbox{div}(\vec{B})-J
%\end{equation}
%onde

%\[
%\begin{aligned}
%A &=t\left[  \beta\mid u \mid^{2}+\alpha p^2 \right]-2p(u\Large\textbf{.}\nabla h) + 2\alpha p\int_{0}^{t}p(x,s)ds -2\beta u(x,0)\Large\textbf{.}\int_{0}^{t}u(x,s)ds \\
%\vec{B}  &= -2\alpha\beta t p u +\alpha p^2 \nabla h -\beta \mid u \mid^{2}\nabla h + 2\beta(u\Large\textbf{.}\nabla h)u -2\alpha\beta\left( \int_{0}^{t} p(x,s)ds \right)u\\
%J &= \beta(\Delta -1)\mid u \mid^2 -2\beta\sum_{i,j=1}^{3}\frac{\partial^2h}{\partial x_i\partial x_j}u_i u_j -\alpha(\Delta h -3)p^2
%\end{aligned}
%\]

The expression above, we make rewrite as
\begin{equation}{\label{Eq_Control}}
0=\frac{\partial A}{\partial t}-\mbox{div}(\vec{B})-J
\end{equation}
where
\[
\begin{aligned}
A &=t\left[  \beta\mid u \mid^{2}+\alpha p^2 \right]-2p(u\Large\textbf{.}\nabla h) + 2\alpha p\int_{0}^{t}p(x,s)ds -2\beta u(x,0)\Large\textbf{.}\int_{0}^{t}u(x,s)ds \\
\vec{B}  &= -2\alpha\beta t p u +\alpha p^2 \nabla h -\beta \mid u \mid^{2}\nabla h + 2\beta(u\Large\textbf{.}\nabla h)u -2\alpha\beta\left( \int_{0}^{t} p(x,s)ds \right)u\\
J &= \beta(\Delta -1)\mid u \mid^2 -2\beta\sum_{i,j=1}^{3}\frac{\partial^2h}{\partial x_i\partial x_j}u_i u_j -\alpha(\Delta h -3)p^2
\end{aligned}
\]

%\begin{remark}
%Note que se considerarmos $h(x)=\frac{1}{2}\mid x-x_0 \mid^2$ para algum $x_0\in I\!\!R$ fixo, então $J=0$. Neste caso (\ref{Eq_Control}) representa uma lei de conservação. Se integramos (\ref{Eq_Control}) em $\Omega_k$ observe que, pela expressão de $B$, precisaremos, para obter boas estimativas, fixar o sinal de $\frac{\partial h}{\partial\eta}$. Isso leva na escolha de $h(x)$ como uma pequena perturbação de $\frac{1}{2}\mid x-x_0 \mid^2$ para algum $x_0\in I\!\!R^3$.
%\end{remark}

\begin{remark}
We have considered  $h(x)=\frac{1}{2}\mid x-x_0 \mid^2$ for some  $x_0\in I\!\!R$ fixed, then $J=0$. In this case (\ref{Eq_Control}) represents a conservation law. Integrating (\ref{Eq_Control}) and $\Omega_k$ , we observe that: in the expression  $B$ need to fix the $\frac{\partial h}{\partial\eta}$ to get good estimations; then,  we chose  $h(x)$ as a small perturbation of $\frac{1}{2}\mid x-x_0 \mid^2$  for some $x_0\in I\!\!R^3$.
\end{remark}

Integrating the identity (\ref{Eq_Control}) in  $\Omega_k \times (0,s)$ and adding in $k$, we have 
\begin{equation}{\label{Partes_1}}
0=\sum_{k=0}^{m}\int_{\Omega_k}\left[ A^k(x,s)- A^k(x,0)  \right]dx - \sum_{k=0}^{m}\int_{0}^{s}\int_{\partial\Omega_k}\vec{B}^{k}{\Large\textbf{.}}\eta d\Gamma - \sum_{k=0}^{m}\int_{\Omega_k}\int_{0}^{s}J^{k} dt dx
\end{equation}
replacing the expression  $A$ in (\ref{Partes_1}), we have that
\[
\begin{aligned}
0 &= \sum_{k=0}^{m}\int_{\Omega_k}\left\{ s\left[ \beta^k \mid u^k \mid^2 + \alpha^k(p^k)^2  \right] - 2p^k(x,s)(u^k\Large\textbf{.}\nabla h) +2\alpha^k p^k\int_{0}^{s}p^{k}(x,\tau)d\tau\right.\\ 
&- \left. 2\beta^k u_{0}^{k}(x)\Large\textbf{.}\int_{0}^{s}u^{k}(x,r)dr +2p_{0}^{k}(u_{0}^{k}\Large\textbf{.}\nabla h) \right\}dx - \sum_{k=0}^{m}\int_{0}^{s}\int_{\partial\Omega_k}\vec{B}^k\Large\textbf{.}\eta d\Gamma_k-\sum_{k=0}^{m}\int_{\Omega_k}\int_{0}^{s}J^{k}dtdx
\end{aligned}
\]
thus,
\begin{equation}{\label{Estim_Varios}}
\begin{aligned}
s\sum_{k=0}^{m}\int_{\Omega_k}\left[ \beta^k \mid u^k \mid^2 + \alpha^k(p^k)^2  \right]  &= 2\sum_{k=0}^{m}\int_{\Omega_k}2p^k(x,s)(u^k{\Large\textbf{.}}\nabla h)dx - 2\sum_{k=0}^{m}\int_{\Omega_k}\alpha^k p^k\int_{0}^{s}p^{k}(x,\tau)d\tau\\ 
&+ 2\sum_{k=0}^{m}\int_{\Omega_k}\beta^k u_{0}^{k}(x){\Large\textbf{.}}\int_{0}^{s}u^{k}(x,r)drdx -2\sum_{k=0}^{m}\int_{\Omega_k}p_{0}^{k}(u_{0}^{k}{\Large\textbf{.}}\nabla h)dx \\
&+ \sum_{k=0}^{m}\int_{0}^{s}\int_{\partial\Omega_k}\vec{B}^k{\Large\textbf{.}}\eta d\Gamma_k+ \sum_{k=0}^{m}\int_{\Omega_k}\int_{0}^{s}J^{k}dtdx
\end{aligned}
\end{equation}
where $p^{k}_{0}=p^{k}(x,0)$, $u^{k}_{0}=u^{k}(x,0)$. 
The proof of the main result is to get the inequality of the observability; we obtain this estimation using the right side of (\ref{Estim_Varios}). Making good presentation of the proof, we show many lemmas

\begin{lemma} 
Given $\{ u, p\}$ regular solution for the problem (\ref{Par_Ecu_Ondas_Multilayer_1})-(\ref{Cond_Multilayer_1}), this was given by the theorem \ref{Sol_Fortes_P1}. Then 
\[
\sum_{k=0}^{m}\int_{\Omega_k}2p^k(x,s)(u^k.\nabla h)dx\leq C_1\sum_{k=0}^{m}\left\{ \beta^k\mid u^k \mid^2+\alpha^k(p^k)^2 \right\}dx
\]
where $C_1=3\max_{k=0,1,\dots,m}\left\{ (\alpha^k)^{-1}, (\beta^k)^{-1} \right\}\max_{x\in\Omega}\mid \nabla h \mid$
\end{lemma}
\begin{proof}
Using the Holder inequality in the first term in the right-hand side of (\ref{Estim_Varios}), we have that
\begin{equation}{\label{Primeiro_Termo}}
	\begin{aligned}
		2\sum_{k=0}^{m}\int_{\Omega_k}p^k(x,s)(u^k.\nabla h)dx &=2\sum_{k=0}^{m}\sum_{i=1}^{3}\int_{\Omega_k}p^k u_{i}^{k}\frac{\partial h}{\partial x_i}dx \leq 2\max_{x\in\Omega}			\mid \nabla h \mid\sum_{k=0}^{m}\sum_{i=1}^{3}\int_{\Omega_k}p^k u_{i}^{k} dx \\
		& \leq 2\max_{x\in\Omega}\mid \nabla h \mid \sum_{k=0}^{m}\sum_{i=1}^{3}\left(\int_{\Omega_k}(p^k)^2 \right)^{1/2}\left(\int_{\Omega_k}(u_i^k)^2 \right)^{1/2}\\	         & \leq \max_{x\in\Omega}\mid \nabla h \mid \sum_{k=0}^{m}\left\{ 3\int_{\Omega_k}(p^k)^2 dx +\int_{\Omega_k}\mid u^k \mid^2 dx \right\}\\	         
		& \leq 3\max_{k=0,1,\dots,m}\left\{ (\alpha^k)^{-1}, (\beta^k)^{-1} \right\}\max_{x\in\Omega}\mid \nabla h \mid \sum_{k=0}^{m}\left\{ \beta^k\mid u^k \mid^2+\alpha^k 				  (p^k)^2 \right\}dx \\
		&= C_1\sum_{k=0}^{m}\left\{ \beta^k\mid u^k \mid^2+\alpha^k(p^k)^2 \right\}dx
	\end{aligned}
\end{equation}
where, $$C_1=3\max_{k=0,1,\dots,m}\left\{ (\alpha^k)^{-1}, (\beta^k)^{-1} \right\}\max_{x\in\Omega}\mid \nabla h \mid$$
\end{proof}

%O segundo termo do lado direito de (\ref{Estim_Varios}), $2\sum_{k=0}^{m}\int_{\Omega_k}\alpha^k p^k\int_{0}^{s}p^{k}(x,\tau)d\tau$, pode-se escrever como

%\begin{equation}{\label{Segundo_Termo}}
%	\begin{aligned}
%		2\sum_{k=0}^{m}\int_{\Omega_k}\alpha^k p^k\int_{0}^{s}p^{k}(x,\tau)d\tau = -\frac{\partial}{\partial s}\sum_{k=0}^{m}\int_{\Omega_k}\alpha^{k}\left[ \int_{0}^{s}p^{k}(x,r)dr  \right]dx
%	\end{aligned}
%\end{equation}

%Para estimar o quarto termo do lado direito de (\ref{Estim_Varios}), $-2\sum_{k=0}^{m}\int_{\Omega_k}p_{0}^{k}(u_{0}^{k}{\Large\textbf{.}}\nabla h)dx$ usamos o fato da energia ser independente do tempo e a estimativa (\ref{Primeiro_Termo}), assim:
%\begin{equation}{\label{Quarto_Termo}}
%	\begin{aligned}
%		-2\sum_{k=0}^{m}\int_{\Omega_k}p_{0}^{k}(u_{0}^{k}{\Large\textbf{.}}\nabla h)dx\leq C_1\sum_{k=0}^{m}			\int_{\Omega_k}\left[  \beta^k\mid u^k\mid^2 +\alpha^k(p^k)^2 \right]
%		^2dx.
%	\end{aligned}
%\end{equation}

%O quinto termo lado direito de (\ref{Estim_Varios}), $\sum_{k=0}^{m}\int_{0}^{s}\int_{\partial\Omega_k}\vec{B}^k.\eta d\Gamma_k$, precisa de uma análise mais cuidadosa.

The second term in the right-hand side of (\ref{Estim_Varios}), $2\sum_{k=0}^{m}\int_{\Omega_k}\alpha^k p^k\int_{0}^{s}p^{k}(x,\tau)d\tau$, we may write as 
\begin{equation}{\label{Segundo_Termo}}
	\begin{aligned}
		2\sum_{k=0}^{m}\int_{\Omega_k}\alpha^k p^k\int_{0}^{s}p^{k}(x,\tau)d\tau = -\frac{\partial}{\partial s}\sum_{k=0}^{m}\int_{\Omega_k}\alpha^{k}\left[ \int_{0}^{s}p^{k}(x,r)dr  \right]dx.
	\end{aligned}
\end{equation}
To estimate the fourth term in the right-hand side of (\ref{Estim_Varios}), $-2\sum_{k=0}^{m}\int_{\Omega_k}p_{0}^{k}(u_{0}^{k}{\Large\textbf{.}}\nabla h)dx$, we used the assumption  the independence of the energy with the time the estimation is (\ref{Primeiro_Termo}), Thus:
\begin{equation}{\label{Quarto_Termo}}
	\begin{aligned}
		-2\sum_{k=0}^{m}\int_{\Omega_k}p_{0}^{k}(u_{0}^{k}{\Large\textbf{.}}\nabla h)dx\leq C_1\sum_{k=0}^{m}			\int_{\Omega_k}\left[  \beta^k\mid u^k\mid^2 +\alpha^k(p^k)^2 \right]
		^2dx.
	\end{aligned}
\end{equation}
The fifth term of right-hand side of (\ref{Estim_Varios}), $\sum_{k=0}^{m}\int_{0}^{s}\int_{\partial\Omega_k}\vec{B}^k.\eta d\Gamma_k$, we need to make an analysis more carefully.

\begin{lemma}{\label{Lema_Interfaces}}
 Given $\{ u, p\}$ a regular solution of the problem (\ref{Par_Ecu_Ondas_ST1})-(\ref{Cond_Multilayer_1}) by  theorem \ref{Sol_Fortes_P1}. For $k=1,2,\dots,m$, we have the following identity
\[
\vec{B}^{k-1}{\Large\textbf{.}}\eta - \vec{B}^{k}{\Large\textbf{.}}\eta = -\frac{\partial h}{\partial\eta}\left\{ \frac{(\alpha^{k-1}-\alpha^k)}{\alpha^{k-1}}\alpha^{k}(p^k)^2 +(\beta^{k-1}-\beta^{k})\frac{\beta^{k}}{\beta^{k-1}}\mid u^k{\Large\textbf{.}}\eta\mid^2 + (\beta^{k-1}-\beta^{k})\mid u^k\times\eta  \mid^2 \right\}
\]
is validated in  $\Gamma_k$, $k=1,2,\dots,m$.
\end{lemma}
\begin{proof}
Using the boundary conditions (\ref{Cond_Multilayer_1}), we have that  $x\in S_1$ , then
\begin{equation}{\label{Quinto_Termo}}
	\begin{aligned}
		\vec{B}{\Large\textbf{.}}\eta &= -\beta\mid u\mid^2 \frac{\partial h}{\partial\eta}+2\beta(u{\Large\textbf{.}}\nabla h)			(u.\eta)\\
		&= -\beta\mid u\mid^2 \frac{\partial h}{\partial\eta}+2\beta\left\{ \mid u 							\mid^2\frac{\partial h}{\partial\eta} +(u\times\eta)(\nabla h\times u)\right\}\\
		&=-\beta\mid u\mid^2 \frac{\partial h}{\partial\eta}+2\beta\mid u \mid^2\frac{\partial 				h}{\partial\eta}\\
		&= \beta\mid u\mid^2 \frac{\partial h}{\partial\eta}.
	\end{aligned}
\end{equation}
and  $x\in S_0$, then 
\[
\vec{B}{\Large\textbf{.}}\eta= \alpha p^2\frac{\partial h}{\partial\eta}-\beta\mid u \mid^2 \frac{\partial h}{\partial\eta}.
\]
Using the interface conditions (\ref{Cond_Multilayer_1}), for $x\in\Gamma_k$ we have the following identity 
\begin{equation}{\label{Aux_interf_B}}
	\begin{aligned}
		\vec{B}^{k-1}{\Large\textbf{.}}\eta - \vec{B}^{k}{\Large\textbf{.}}\eta &= -2\alpha^{k-1}\beta^{k-1}+p^{k-1}(u^{k-1}.				\eta)+\alpha^{k-1}(p^{k-1})^2 \frac{\partial h}{\partial\eta}-\beta^{k-1}\mid u^{k-1}\mid^2 \frac{\partial h}{\partial\eta}+\\
		&+ 2\beta^{k-1}(u^{k-1}.\nabla h)(u^{k-1}{\Large\textbf{.}}\eta)-2\alpha^{k-1}\beta^{k-1}\left( \int_{0}^{t}p^{k-1}(x,s)ds \right)(u^{k-1}{\Large\textbf{.}}\eta)+ \\
		&- \alpha^{k}(p^{k})^2 \frac{\partial h}{\partial\eta} +\beta^{k}\mid u^k\mid^2\frac{\partial h}{\partial\eta}-2\beta^k(u^k{\Large\textbf{.}}\nabla h)(u^k{\Large\textbf{.}}\eta)+2\alpha^k\beta^k \left( \int_{0}^{t}p^k(x,s)ds \right)(u^k{\Large\textbf{.}}\eta)\\	
		&= \alpha^{k-1}(p^{k-1})^2 \frac{\partial h}				{\partial\eta}-\beta^{k-1}\mid u^{k-1}\mid^2 \frac{\partial h}{\partial\eta}+2\beta^{k-1}(u^{k-1}{\Large\textbf{.}}\nabla h)(u^{k-1}{\Large\textbf{.}}\eta) -\\
		& \alpha^{k}(p^{k})^2 \frac{\partial h}{\partial\eta}+\beta^k \mid u^k \mid^2 \frac{\partial h}{\partial\eta}-2\beta^k(u^k{\Large\textbf{.}}\nabla h)(u^k{\Large\textbf{.}}\eta)
	\end{aligned}
\end{equation}
else,
\begin{equation}{\label{Aux_interf}}
	\begin{aligned}
		\alpha^{k-1}(p^{k-1})^2\frac{\partial h}{\partial\eta} - \alpha^k(p^k)^2\frac{\partial h}			{\partial\eta} &= \frac{\left( \alpha^{k-1}p^{k-1} \right)^2}{\alpha^{k-1}} 					\frac{\partial h}{\partial\eta}\\
		&= \left( \frac{1}{\alpha^{k-1}}-\frac{1}{\alpha^{k}} \right)(\alpha^{k}p^{k})^2 					\frac{\partial h}{\partial\eta} \\
		&=-(\alpha^{k-1}-\alpha^{k})\frac{\alpha^{k}}{\alpha^{k-1}}(p^k)^2 \frac{\partial h}				{\partial\eta}
	\end{aligned}
\end{equation}
For $\mid \eta \mid=1$ is validated
\begin{equation}{\label{Aux_interf_1}}
 \mid u \mid^2= \mid (u.\eta) \mid^2+\mid u\times\eta \mid^2 
\end{equation}
Substituting (\ref{Aux_interf_1}) in (\ref{Aux_interf_B}), we have that
\begin{equation}{\label{Aux_interf_2}}
	\begin{aligned}
		\beta^k \mid u^k \mid^2 \frac{\partial h}{\partial\eta}-\beta^{k-1} \mid u^{k-1} \mid^2\frac{\partial h}{\partial\eta} &= \beta^k \left( \mid (u^k.\eta) \mid^2 + \mid (u^k\times\eta) \mid^2 \right)\frac{\partial h}{\partial\eta} - \beta^{k-1} \left( \mid(u^{k-1}{\Large\textbf{.}}\eta)\mid^2 +\right.\\
		 &+\left. \mid (u^{k-1}\times\eta) \mid^2 \right)\frac{\partial h}{\partial\eta} \\
		 &=\left(  \beta^k \mid u^k{\Large\textbf{.}}\eta \mid^2-\beta^{k-1}\mid u^{k-1}{\Large\textbf{.}}\eta \mid^2  \right)\frac{\partial h}{\partial\eta} + \left( \beta^k \mid u^k\times\eta \mid^2- \right.\\
		 &\left.\beta^{k-1}\mid  u^{k-1}\times\eta  \mid^2 \right)\frac{\partial h}{\partial\eta}\\
		 &= \left[  \frac{(\beta^k \mid u^k{\Large\textbf{.}}\eta \mid)^2}{\beta^k}- \frac{(\beta^{k-1}\mid 					u^{k-1}{\Large\textbf{.}}\eta \mid)^2}{\beta^{k-1}} \right]\frac{\partial h}{\partial\eta} +\\ 
		 & (\beta^{k}-\beta^{k-1})\mid u^k\times\eta \mid^2\frac{\partial h}{\partial\eta} \\
		 &= \left( \frac{1}{\beta^{k}}- \frac{1}{\beta^{k-1}} \right)(\beta^{k}\mid u^k{\Large\textbf{.}}\eta\mid)^2 + (\beta^{k}-\beta^{k-1})\mid  u^k\times\eta  \mid^2\frac{\partial h}{\partial\eta}\\
		 &= (\beta^{k-1}-\beta^{k})\frac{\beta^{k}}{\beta^{k-1}}\mid u^k.\eta \mid^2\frac{\partial h}{\partial\eta}- (\beta^{k-1}-\beta^{k})\mid  u^k\times\eta  					\mid^2\frac{\partial h}{\partial\eta}
	\end{aligned}
\end{equation}
Finally,
\begin{equation}{\label{Aux_interf_2}}
	\begin{aligned}
		2\beta^{k-1}(u^{k-1}{\Large\textbf{.}}\nabla h)(u^{k-1}.\eta)-2\beta^k (u^k.\nabla h)(u^k.\eta) &= 				2\beta^{k-1}\mid u^{k-1}{\Large\textbf{.}}\eta  \mid^2 -2\beta^k\mid u^k{\Large\textbf{.}}\eta \mid^2\frac{\partial h}{\partial\eta}\\
		&=2\frac{\left(\beta^{k-1}\mid u^{k-1}{\Large\textbf{.}}\eta \mid \right)^2}{\beta^{k-1}}\frac{\partial 			  h}{\partial\eta}-2\frac{\left( \beta^{k}\mid u^{k}{\Large\textbf{.}}\eta \mid \right)^2}{\beta^{k}}\frac{\partial h}{\partial\eta}\\
		&= 2\left(  \frac{1}{\beta^{k-1}}-\frac{1}{\beta^{k}} \right)(\beta^{k}\mid u^{k}{\Large\textbf{.}}\eta\mid)^2 \frac{\partial h}{\partial\eta} \\
		&=-2(\beta^{k-1}-\beta^{k})\frac{\beta^{k}}{\beta^{k-1}}\mid u^{k}{\Large\textbf{.}}\eta\mid^2\frac{\partial h}{\partial\eta}.
	\end{aligned}
\end{equation}
Substituting (\ref{Aux_interf}), (\ref{Aux_interf_1}) and (\ref{Aux_interf_2}) in (\ref{Aux_interf_B}), we have that
\begin{equation}{\label{Aux_interf_B1}}
	\begin{aligned}
		\vec{B}^{k-1}.\eta - \vec{B}^{k}.\eta &= -(\alpha^{k-1}-\alpha^k)\frac{\alpha^k}{\alpha^{k-1}}(p^k)^2\frac{\partial h}{\partial\eta}-(\beta^{k-1}-\beta^k)						\frac{\beta^k}{\beta^{k-1}}\mid u^k{\Large\textbf{.}}\eta  \mid^2\frac{\partial h}{\partial\eta}-\\
		& (\beta^{k-1}-\beta^{k})\mid u^k\times\eta  \mid^2 \frac{\partial h}{\partial\eta}\\
		&= -\frac{\partial h}{\partial\eta}\left\{ \frac{(\alpha^{k-1}-\alpha^k)}{\alpha^{k-1}}			\alpha^{k}(p^k)^2 +(\beta^{k-1}-\beta^{k})\frac{\beta^{k}}{\beta^{k-1}}\mid u^k{\Large\textbf{.}}\eta\mid^2 + (\beta^{k-1}-\beta^{k})\mid u^k\times\eta  \mid^2 \right\}.
	\end{aligned}
\end{equation}
\end{proof}

Now, we estimate the sixth term, $\sum_{m}^{k=0}\int_{\Omega_k}\int_{0}^{s}J^k dt dx$. Remember that,
\begin{equation}{\label{Termo_6}}
\sum_{m}^{k=0}\int_{\Omega_k}\int_{0}^{s}J^k dt dx=\sum_{k=0}^{m}\int_{\Omega_k}\int_{0}^{s}\left\{ \beta^k(\Delta h-1)\mid u^k \mid^2-2\beta^k \sum_{i,j=1}^{m}\frac{\partial^2 h}{\partial x_i \partial x_j}u_i^k u_j^k-\alpha^k(\Delta h-3)(p^k)^2 \right\}
\end{equation}
To estimate  (\ref{Termo_6}), we choose a function $h$ as:
\begin{equation}{\label{Fun_h}}
h(x)=\frac{1}{2}\mid  x-x_0  \mid^2 +\delta_0 \Phi(x)
\end{equation}
where  $x_0\in \sigma_1$ and $\Phi$ satisfy 
\begin{equation}{\label{Funcao_Phi}}
	\left\{\begin{aligned}
		\Delta\Phi &=1\quad\mbox{en}\quad \Omega\\
		\frac{\partial\Phi}{\partial\eta} &= 2\frac{\mbox{Vol}(\Omega)}{\mbox{area}(S_0)},\quad			\mbox{in},\quad S_0 \\
		\frac{\partial\Phi}{\partial\eta} &= -\frac{\mbox{Vol}(\Omega)}{\mbox{area}(S_1)},\quad			\mbox{in},\quad S_1 
	\end{aligned}\right.
\end{equation}
\begin{remark}{\label{Prop_Mu}}
Given $\mu=\mu(\Omega)$, by
\begin{equation}{\label{Mu}}
	\mu(\Omega)=\inf_{\begin{aligned}
	x\in\Omega \\
	\mid\xi\mid=1
	\end{aligned}}2\sum_{i,j=1}^{3}\frac{\partial^2\Phi(x)}{\partial x_i \partial x_j}\xi_i 		\xi_j
\end{equation}
we may observe that, considering  $\xi=(1,0,0), \xi=(0,1,0), \xi=(0,0,1)$ , have that
\[
\mu(\Omega)\leq 2\frac{\partial^2\Phi(x)}{\partial x_1^2},\quad  \mu(\Omega)\leq 2\frac{\partial^2\Phi(x)}{\partial x_2^2}, \quad \mu(\Omega)\leq 2\frac{\partial^2\Phi(x)}{\partial x_3^2}
\]
that, adding the last expressions we have that
\[
3\mu(\Omega)\leq 2\Delta\Phi,\Longrightarrow  \mu(\Omega)\leq \frac{2}{3}
\]
the expression of $h(x)$,  then 
\[
\frac{\partial^2 h(x)}{\partial x_i\partial x_j}=\delta_{ij}+\delta_0\frac{\partial^2\Phi(x)}{\partial x_i\partial x_j}
\]
and $\Delta h=3+\delta_0$. 
\end{remark}

%%

%\begin{lemma}
% Seja $\{ u, p\}$ solução regular do problema (\ref{Par_Ecu_Ondas_ST1})-(\ref{Cond_Multilayer_1}), dada pelo Teorema \ref{Sol_Fortes_P1}. Escolhendo $h$ como em (\ref{Fun_h}) temos
% \[
%\sum_{k=0}^{m}\int_{\Omega_k}\int_{0}^{s}J^k dt dx \leq \delta_0(1-\mu(\Omega))\sum_{k=0}^{m}\int_{\Omega_k}\int_{0}^{s}\left\{ \beta^k \mid u^k\mid^2+ \alpha^k(p^k)^2 \right\}dtdx
% \]
% para qualquer $\delta_0>0$
%\end{lemma}

%\begin{proof}

%Usando (\ref{Fun_h}), (\ref{Funcao_Phi}), (\ref{Mu}) e a observação (\ref{Prop_Mu}), temos

%\begin{equation}{\label{Estima_Termo_6}}
%	\begin{aligned}
%		\sum_{k=0}^{m}\int_{\Omega_k}\int_{0}^{s}J^k dt dx &= \sum_{k=0}^{m}\int_{\Omega_k}				\int_{0}^{s}\left\{ \beta^k(2+\delta_0)\mid u^k \mid^2-2\beta^k \sum_{i,j=1}^{m}				\left( \delta_{ij}+\delta_0\frac{\partial^2\Phi}{\partial x_i\partial x_j}\right) u_i^k 			u_j^k-\alpha^k\delta_0(p^k)^2\right\}\\
%		&\leq \sum_{k=0}^{m}\int_{\Omega_k}\int_{0}^{s}\left\{ \beta^k(2+\delta_0)\mid u^k 				 \mid^2-2\beta^k \mid u^k\mid^2- \beta^k\mu(\Omega)\delta_0\mid u^k\mid^2 -\alpha^k				 \delta_0(p^k)^2\right\}\\
%		&\leq \sum_{k=0}^{m}\int_{\Omega_k}\int_{0}^{s}\left\{ \delta_0\beta^k\mid u^k 					 \mid^2-\mu(\Omega)\delta_0\beta^k \mid u^k\mid^2-\alpha^k\delta_0(p^k)^2\right\}dtdx\\
%		&\leq \sum_{k=0}^{m}\int_{\Omega_k}\int_{0}^{s}\left\{  \delta_0(1-\mu(\Omega))\beta^k
%		 \mid u^k \mid^2 \right\}dt dx\\
%		&\leq \delta_0(1-\mu(\Omega))\sum_{k=0}^{m}\int_{\Omega_k}\int_{0}^{s}\left\{ \beta^k 			 \mid u^k\mid^2+ \alpha^k(p^k)^2 \right\}dtdx
%	\end{aligned}
%\end{equation}

%\end{proof}

\begin{lemma}
 Given $\{ u, p\}$ a regular solution of the problem (\ref{Par_Ecu_Ondas_ST1})-(\ref{Cond_Multilayer_1}) by theorem \ref{Sol_Fortes_P1}. Choosing $h$ as (\ref{Fun_h}), we have that 
 \[
\sum_{k=0}^{m}\int_{\Omega_k}\int_{0}^{s}J^k dt dx \leq \delta_0(1-\mu(\Omega))\sum_{k=0}^{m}\int_{\Omega_k}\int_{0}^{s}\left\{ \beta^k \mid u^k\mid^2+ \alpha^k(p^k)^2 \right\}dtdx
 \]
 for any $\delta_0>0$
\end{lemma}
\begin{proof}
Using (\ref{Fun_h}), (\ref{Funcao_Phi}), (\ref{Mu}) and the observation (\ref{Prop_Mu}), we have that 
\begin{equation}{\label{Estima_Termo_6}}
	\begin{aligned}
		\sum_{k=0}^{m}\int_{\Omega_k}\int_{0}^{s}J^k dt dx &= \sum_{k=0}^{m}\int_{\Omega_k}				\int_{0}^{s}\left\{ \beta^k(2+\delta_0)\mid u^k \mid^2-2\beta^k \sum_{i,j=1}^{m}				\left( \delta_{ij}+\delta_0\frac{\partial^2\Phi}{\partial x_i\partial x_j}\right) u_i^k 			u_j^k-\alpha^k\delta_0(p^k)^2\right\}\\
		&\leq \sum_{k=0}^{m}\int_{\Omega_k}\int_{0}^{s}\left\{ \beta^k(2+\delta_0)\mid u^k 				 \mid^2-2\beta^k \mid u^k\mid^2- \beta^k\mu(\Omega)\delta_0\mid u^k\mid^2 -\alpha^k				 \delta_0(p^k)^2\right\}\\
		&\leq \sum_{k=0}^{m}\int_{\Omega_k}\int_{0}^{s}\left\{ \delta_0\beta^k\mid u^k 					 \mid^2-\mu(\Omega)\delta_0\beta^k \mid u^k\mid^2-\alpha^k\delta_0(p^k)^2\right\}dtdx\\
		&\leq \sum_{k=0}^{m}\int_{\Omega_k}\int_{0}^{s}\left\{  \delta_0(1-\mu(\Omega))\beta^k
		 \mid u^k \mid^2 \right\}dt dx\\
		&\leq \delta_0(1-\mu(\Omega))\sum_{k=0}^{m}\int_{\Omega_k}\int_{0}^{s}\left\{ \beta^k 			 \mid u^k\mid^2+ \alpha^k(p^k)^2 \right\}dtdx
	\end{aligned}
\end{equation}
\end{proof}

To estimate the rest of the terms of (\ref{Estim_Varios}), $2\sum_{k=0}^{m}\int_{\Omega_k}\beta^k u_0^k (x){\Large\textbf{.}}\int_{0}^{s}u^{k}(x,r)drdx$, considering a hypothesis about initial condition $u_0(x)$, we assumed that, it satisfied the following system:
\begin{equation}{\label{Hipo_U0}}
	\left\{ \begin{aligned}
		u_{0}^{k-1}.\eta &= u_{0}^{k}.\eta\quad\mbox{in}\quad\Gamma_k \\
		u_{0}^{k-1}\times\eta &= u_{0}^{k}\times\eta\quad\mbox{in}\quad\Gamma_k	\\
		u_{0}^{k} &=\nabla l^{k}(x)\quad\mbox{in}\quad\Gamma_k \\
		 l^{k} &\in H^{2}(\Omega_k)\quad \mbox{and}\quad h=0\quad\mbox{in}\quad S_1
	\end{aligned}\right.
\end{equation}
\begin{remark}
The hypothesis have been made  in  (\ref{Hipo_U0}), inclusive the solution of the problem satisfy the properties of the lemma (\ref{Prop_A_j}), these are necessaries because the domain is conexo.
\end{remark}
\begin{lemma}
Given  $\{ u, p\}$  regular solution of the problem (\ref{Par_Ecu_Ondas_Multilayer_1})-(\ref{Cond_Multilayer_1}) by theorem \ref{Sol_Fortes_P1},  and the initial condition $u_0$ that satisfied (\ref{Hipo_U0}). Then  
\[
2\sum_{k=0}^{m}\int_{\Omega_k}\beta^k u_0^k (x)\int_{0}^{s}u^{k}(x,r)drdx = 2\sum_{k=1}^{m}\int_{\Gamma_k}l^k(p^k(x,s)-p_{0}^{k}(x))dx
\]
\end{lemma}
\begin{proof}
Using the hypothesis  $l$ and  the boundary condition in $S_0$, we have that 
\begin{equation}{\label{Estim_termo3}}
	\begin{aligned}
		2\sum_{k=0}^{m}\int_{\Omega_k}\beta^k u_0^k (x)\int_{0}^{s}u^{k}(x,r)drdx &=2\sum_{k=0}^{m}\int_{\Omega_k}\int_{0}^{s}\beta^k\nabla l^k{\Large\textbf{.}}u^k \\
		&= 2\sum_{k=0}^{m}\int_{0}^{s}\left\{  -\int_{\Omega_k}\beta^k l^k\mbox{div}(u^k)+ 				\int_{\partial\Omega_k}\beta^k l^k (u^k{\Large\textbf{.}}\eta) \right\}\\
		&= 2\sum_{k=0}^{m}\int_{0}^{s}\int_{\Omega_k}l^k\frac{\partial p^k}{\partial t} +2\sum_{k=0}^{m}\int_{0}^{s}\int_{\partial\Omega_k}\beta^k l^k (u^k{\Large\textbf{.}}\eta)\\
		&=2\sum_{k=0}^{m}\int_{0}^{s}\int_{\Omega_k}l^k\frac{\partial p^k}{\partial t} + 2\int_{0}^{s}\int_{S_1}\beta l(u{\Large\textbf{.}}\eta) +\\
		& 2\sum_{k=1}^{m}\int_{0}^{s}\int_{\Gamma_k}\left[ \beta^{k-1}l^{k-1}(u^{k-1}{\Large\textbf{.}}\eta)-\beta^{k}l^{k}(u^{k}.\eta) \right]+ 2\int_{0}^{s}\int_{S_0}\beta l(u{\Large\textbf{.}}\eta)\\
		&= 2\sum_{k=1}^{m}\int_{\Gamma_k}l^k(p^k(x,s)-p_{0}^{k}(x)) + \\ 
		& 2\sum_{k=1}^{m}\int_{0}^{s}\int_{\Gamma_k}\left[  \beta^{k-1}l^{k-1}(u^{k-1}{\Large\textbf{.}}\eta)-\beta^{k}l^{k}(u^{k}{\Large\textbf{.}}\eta) \right]
	\end{aligned}
\end{equation}
by the hypothesis made in the initial condition (\ref{Hipo_U0}), we have that 
\begin{equation}{\label{Estim_termo3_0}}
	\begin{aligned}
		\mbox{De},\quad u^{k-1}_{0}{\Large\textbf{.}}\eta=u^{k}_{0}.\eta\quad\mbox{in}\quad\Gamma_k & 					\Longrightarrow \nabla l^{k-1}{\Large\textbf{.}}\eta= \nabla l^{k}.\eta\quad\mbox{in}\quad\Gamma_k \nonumber\\
		& \Longrightarrow \nabla(l^{k-1}-l^{k}){\Large\textbf{.}}\eta = 0 \quad\mbox{in}\quad\Gamma_k \\
		& \Longrightarrow \nabla(l^{k-1}-l^{k}){\Large\textbf{.}}\eta \perp \eta\quad\mbox{in}\quad\Gamma_k.
	\end{aligned}
\end{equation}
\begin{equation}{\label{Estim_termo3_1}}
	\begin{aligned}
		\mbox{De},\quad u^{k-1}_{0}\times\eta=u^{k}_{0}\times\eta\quad\mbox{in}\quad\Gamma_k &\Longrightarrow \nabla l^{k-1}\times\eta= \nabla l^{k}\times\eta\quad\mbox{in}\quad				\Gamma_k\\
		& \Longrightarrow \nabla(l^{k-1}-l^{k})\times\eta = 0 \quad\mbox{in}\quad\Gamma_k \\
		& \Longrightarrow \nabla(l^{k-1}-l^{k})\times\eta \mbox{//} \eta\quad\mbox{in}\quad\Gamma_k.
	\end{aligned}
\end{equation}
thereby, $\nabla(l^{k-1}-l^k)=0$ in $\Gamma_k$, it implies that  $l^{k-1}-l^k=\mbox{C}$ in $\Gamma_k$, $\mbox{C}=\mbox{constant}$.
Substituting (\ref{Estim_termo3_1}) in the second tern on the right-hand side of (\ref{Estim_termo3}), then 
\begin{equation}{\label{Estim_termo3_2}}
	\begin{aligned}
		2\sum_{k=1}^{m}\int_{0}^{s}\int_{\Gamma_k}\left[  \beta^{k-1}l^{k-1}(u^{k-1}.\eta)- \beta^{k}l^{k}(u^{k}.\eta)\right] &= 2\sum_{k=1}^{m}\int_{0}^{s}\int_{\Gamma_k} (l^{k-1}-l^{k})\beta^k(u^k.\eta)\\
		&= 2\mbox{C}\sum_{k=1}^{m}\int_{0}^{s}\int_{\Gamma_k}\beta^k(u^k.\eta)\\
		&= 2\int_{0}^{s}\int_{S_0}\beta(u.\eta)=0
	\end{aligned}
\end{equation}
and following substituting (\ref{Estim_termo3_2}) in (\ref{Estim_termo3}).
\end{proof}
Substituting the obtained estimations, (\ref{Primeiro_Termo}), (\ref{Segundo_Termo}), (\ref{Quarto_Termo}), (\ref{Quinto_Termo}),(\ref{Termo_6}), (\ref{Estim_termo3_2}) in (\ref{Estim_Varios}), then ,
\begin{equation}{\label{Estim_varios_1}}
	\begin{aligned}
		s\sum_{k=0}^{m}\int_{\Omega_k}\left[ \beta^k \mid u^k \mid^2 + \alpha^k(p^k)^2  				\right]dx &= 2C_1s\sum_{k=0}^{m}\int_{\Omega_k}\left[ \beta^k \mid u^k \mid^2 + 				\alpha^k(p^k)^2\right]dx - \\ 
		&\frac{\partial}{\partial s}\sum_{k=0}^{m}\alpha^k\left[ p^k(x,r)dr \right]^2 dx 
		+ \int_{0}^{s}\beta\mid u \mid^2\frac{\partial h}{\partial\eta}+\\
		&\sum_{k=0}^{m}\int_{0}^{s}\int_{\Gamma_k}\left[ \vec{B}^{k-1}.\eta - \vec{B}^{k}.				 \eta \right]+ \int_{0}^{s}\int_{S_0}\left( \alpha p^2\frac{\partial h}{\partial\eta} 			-\beta\mid u\mid^2 \frac{\partial h}{\partial\eta}\right) +\\
		& \delta_0(1-\mu(\Omega))\sum_{k=0}^{m}\int_{0}^{s}\int_{\Gamma_k}\left[ \beta^k \mid 			u^k \mid^2 + \alpha^k(p^k)^2 \right] +\\
		& 2\sum_{k=0}^{m}\int_{\Omega_k}l^k(p^k(x,s)-p^k_0(x))dx
	\end{aligned}
\end{equation}
Integrating  (\ref{Estim_varios_1}) in $(0,T)$ and using the independence of energy of the model with the time, we have that 
\begin{equation}{\label{Estim_varios_2}}
	\begin{aligned}
		& \frac{T}{2}\left[ 1-\delta_0(1-\mu(\Omega))\right]\sum_{k=0}^{m}\int_{\Omega_k} 				  \left[\beta^k \mid u^k \mid^2 + \alpha^k(p^k)^2 \right]dx \leq 2C_1 T\sum_{k=0}^{m}			  \int_{\Omega_k}\left[\beta^k \mid u^k \mid^2 + \alpha^k(p^k)^2 \right]dx-\\
		&\sum_{k=0}^{m}\int_{\Omega_k}\alpha^k \left[ \int_{0}^{T}p^k(x,r) \right]^2 dx+				\int_{0}^{T}\int_{0}^{s}\int_{S_1}\beta\mid u\mid^2 \frac{\partial h}{\partial\eta}+			\sum_{k=0}^{m}\int_{0}^{T}\int_{0}^{s}\int_{\Gamma_k}\left[ \vec{B}^{k-1}.\eta - 				\vec{B}^{k}.\eta \right] + \\
		& \int_{0}^{T}\int_{0}^{s}\int_{S_0}\left( \alpha p^2\frac{\partial h}{\partial\eta} 			  - \beta\mid u \mid^2\frac{\partial h}{\partial\eta}  \right) + 2\sum_{k=0}^{m}				\int_{0}^{T}\int_{\Omega_k}l^k (p^k(x,s)-p^{k}_{0}(x))
	\end{aligned}
\end{equation}
Now, we need the hypothesis in the domain $\Omega$. Given $\delta_0>0$ such that, some  $x_0\in \sigma_1$, we have
\begin{equation}{\label{Hipo_Omega}}
	\left\{ \begin{aligned}
		&\delta_{0}(1-\mu(\Omega)) < 1, \\
		&(x-x_0).\eta \geq -2\delta_0\frac{\mbox{Vol}(\Omega)}{\mbox{area}(S_0)},\quad					\mbox{for}\quad x\in S_0\\
		&(x-x_0).\eta \leq \delta_0\frac{\mbox{Vol}(\Omega)}{\mbox{area}(S_1)}, \quad					\mbox{for}\quad x\in S_1\\
        &(x-x_0).\eta+\delta_0\frac{\partial \Phi}{\partial\eta} \geq 0, \quad \forall x	\in			\Gamma_k, \quad k=1,2,\dots,m
	\end{aligned}\right.
\end{equation}

\begin{remark}
The hypothesis made in (\ref{Hipo_Omega}), These are true when $\delta_0=0$ satisfied the surf of kind "star-shaped".
\end{remark}
Using (\ref{Hipo_Omega}) in $S_1$, we have
\begin{equation}{\label{Hipo_S1}}
	\frac{\partial h}{\partial\eta}=\nabla h.\eta=(x-x_0).\eta + \delta_0\frac{\partial\Phi}		{\partial\eta}=(x-x_0).\eta-\delta_0\frac{\mbox{Vol}(\Omega)}{\mbox{area}S_1}\leq 0
\end{equation}
and, for $x\in S_0$,
\begin{equation}{\label{Hipo_S0}}
	\frac{\partial h}{\partial\eta}=\nabla h.\eta=(x-x_0).\eta + \delta_0\frac{\partial\Phi}		{\partial\eta}=(x-x_0).\eta+2\delta_0\frac{\mbox{Vol}(\Omega)}{\mbox{area}S_0}\geq 0
\end{equation}
Substituting (\ref{Hipo_S0}) and (\ref{Hipo_S1}) in (\ref{Estim_varios_2}), then
\begin{equation}
	\int_{0}^{T}\int_{0}^{s}\int_{S_1}\beta \mid u \mid^2\frac{\partial h}{\partial\eta} \leq 0
\end{equation}
and
\begin{equation}
   -\int_{0}^{T}\int_{0}^{s}\int_{S_0}\beta \mid u \mid^2\frac{\partial h}{\partial\eta} \leq 0
\end{equation}
moreover, if the coefficients  $\alpha^{k},\beta^{k}$ satisfied 
\begin{eqnarray}{\label{Mono_coef}}
\alpha^{k-1} &\leq & \alpha^{k} \\\nonumber
\beta^{k-1} &\leq & \beta^{k}\nonumber
\end{eqnarray}
then, the fourth condition in (\ref{Hipo_Omega}), (\ref{Mono_coef}) and the lemma (\ref{Lema_Interfaces}) we have that 
\begin{equation}{\label{Aux_Interf}}
	\sum_{k=1}^{m}\int_{0}^{T}\int_{0}^{s}\int_{\Gamma_k}\left[ \vec{B}^{k-1}.\eta -\vec{B}      	^{k}.\eta  \right]\leq 0
\end{equation}
Using that the associate energy to the system (\ref{Par_Ecu_Ondas_ST1})-(\ref{Par_Ecu_Ondas_Multilayer_1}); with not dependency of the time, we prove the next lemma
\begin{lemma}{\label{Lemma_LP}}
Given a regular solution $\{ u, p\}$  of the  problem (\ref{Par_Ecu_Ondas_ST1}) by theorem \ref{Sol_Fortes_P1},  and the initial condition $u_0$ that satisfied (\ref{Hipo_U0}). Then 
\[
\begin{aligned}
2 \sum_{k=0}^{m}\int_{0}^{T}\int_{\Omega_k}l^k (p^k(x,s)-p^{k}_{0}(x))dx ds &\leq \sum_{k=0}^{m}\int_{\Omega_k}\left[ \int_{0}^{T}p^k(x,s) \right]^2 dx + \\
& (C_3+C_4T)\sum_{k=0}^{m}\int_{\Omega_k}\left[ \beta^k\mid u^k\mid^2 +\alpha^k(p^k)^2 \right]dx
\end{aligned}
\]
where $C_3=C_2\max_{k}\left\{ (\alpha^k \beta^k)^{-1}  \right\}$ and  $C_4=\max_{k}\left\{ C_2(\beta^k)^{-1}, (\alpha^k)^{-1}  \right\}$
\end{lemma}
%Em (\ref{Estim_varios_2}) estimando $2\sum_{k=0}^{m}\int_{0}^{T}\int_{\Omega_k}l^k (p^k(x,s)-p^{k}_{0}(x))$ , usando que a energia não depende do tempo, temos
\begin{proof}
\begin{equation}{\label{Aux_Todos}}
	 \begin{aligned}
	 &2 \sum_{k=0}^{m}\int_{0}^{T}\int_{\Omega_k}l^k (p^k(x,s)-p^{k}_{0}(x))=2\sum_{k=0}^{m}			\int_{0}^{T}\int_{\Omega_k}\mid l^k p^k(x,s)\mid +2\sum_{k=0}^{m}\int_{0}^{T}					\int_{\Omega_k}\mid l^k p^k_0(x)\mid \\
		&2\leq \sum_{k=0}^{m}\int_{\Omega_k} \mid l^k(x)\mid \int_{0}^{T} \mid l^k p^k(x,t)\mid 		+ 2\sum_{k=0}^{m}\int_{0}^{T}\left( \int_{\Omega_k} \mid l^k(x)\mid^2 \right)^{1/				2}\left( \int_{\Omega_k} \mid p^k(x,s)\mid^2 \right)^{1/2}\\
		&\leq 2 \sum_{k=0}^{m}\left[\int_{\Omega_k} (l^k(x))^2 \right]^{1/2}\left[ 						 \int_{\Omega_k}\left[  p^k(x,s) \right]^2\ dx \right]^{1/2} + \sum_{k=0}^{m}\int_{0}			 ^{T}\int_{\Omega_k}\left\{  (l^k)^2 +(p^k)^2  \right\}\\
		&\leq \sum_{k=0}^{m}\left\{ (\alpha^{k-1})^{-1}\int_{\Omega_k}(l^k(x))^2+\alpha^k				\int_{\Omega_k}\left[ \int_{0}^{T}p^k(x,s) \right]^2 dx \right\} + \sum_{k=0}^{m}				\int_{0}^{T}\int_{\Omega_k}\left\{ (l^k)^2 + (p^k)^2 \right\}\\
		&\leq C_2\sum_{k=0}^{m}\int_{\Omega_k}(\alpha^k)^{-1}\mid \nabla l^k \mid^2 + 			        \sum_{k=0}^{m}\alpha^k\left[ \int_{0}^{T}p^{k}(x,s) \right]^2 dx +C_2\sum_{k=0}^{m}				\int_{0}^{T}\int_{\Omega_k}\mid \nabla l^k \mid^2 + \sum_{k=0}^{m}\int_{0}^{T}					\int_{\Omega_k}(p^k)^2 \\
		&\leq C_2\sum_{k=0}^{m}\int_{\Omega_k}(\alpha^k)^{-1}\mid u_{0}^{k}(x) \mid^2 + 			        \sum_{k=0}^{m}\alpha^k\left[ \int_{0}^{T}p^{k}(x,s) \right]^2 dx +C_2\sum_{k=0}^{m}				\int_{0}^{T}\int_{\Omega_k}\mid \nabla l^k \mid^2 + \sum_{k=0}^{m}\int_{0}^{T}					\int_{\Omega_k}(p^k)^2 \\
		&\leq C_2\max_{k}\{(\alpha^k\beta^k)^{-1}\}\sum_{k=0}^{m}\int_{\Omega_k}\left[ \beta^k			 \mid u^k\mid^2 +\alpha^k(p^k)^2 \right] + \sum_{k=0}^{m}\alpha^k\left[ \int_{0}			     ^{T}p^{k}(x,s) \right]^2 dx +\\
		& C_2\sum_{k=0}^{m}\int_{\Omega_k}(\beta^k)^{-1}\int_{0}^{T}\int_{\Omega_k}\left[ 				\beta^k \mid u^k\mid^2 +\alpha^k(p^k)^2 \right] + \sum_{k=0}^{m}(\alpha^k)^{-1}\int_{0}			^{T}\int_{\Omega_k}\left[ \beta^k\mid u^k\mid^2 +\alpha^k(p^k)^2 \right]\\
		& \leq C_2\max_{k}\{(\alpha^k\beta^k)^{-1}\}\sum_{k=0}^{m}\int_{\Omega_k}\left[ \beta^k			 \mid u^k\mid^2 +\alpha^k(p^k)^2 \right] + \sum_{k=0}^{m}\alpha^k\left[ \int_{0}			     ^{T}p^{k}(x,s) \right]^2 dx +\\
		& \max_{k}\{C_2(\beta^k))^{-1},(\alpha^k)^{-1}\}T \sum_{k=0}^{m}\int_{\Omega_k}\left[ 			   \beta^k\mid u^k\mid^2 +\alpha^k(p^k)^2 \right]dx \\
		&\leq \sum_{k=0}^{m}\int_{\Omega_k}\left[ \int_{0}^{T}p^k(x,s) \right]^2 dx + (C_3 				 +C_4T)\sum_{k=0}^{m}\int_{\Omega_k}\left[ \beta^k\mid u^k\mid^2 +\alpha^k(p^k)^2 				 \right]dx
	\end{aligned}
\end{equation} $C_3=C_2\max_{k}\left\{ (\alpha^k \beta^k)^{-1}  \right\}$ and $C_4=\max_{k}\left\{ C_2(\beta^k)^{-1}, (\alpha^k)^{-1}  \right\}$
where $C_3=C_2\max_{k}\left\{ (\alpha^k \beta^k)^{-1}  \right\}$ and $C_4=\max_{k}\left\{ C_2(\beta^k)^{-1}, (\alpha^k)^{-1}  \right\}$
\end{proof}
Substituting in (\ref{Estim_varios_2}), and using the $\frac{\partial h}{\partial\eta}\geq 0$  in $S_0$,  we have that 
\begin{equation}{\label{Ultima_Estim_0}}
	\begin{aligned}
		& \frac{T}{2}\left[ 1-\delta_0(1-\mu(\Omega))\right]\sum_{k=0}^{m}\int_{\Omega_k} 				  \left[\beta^k \mid u^k \mid^2 + \alpha^k(p^k)^2 \right]dx \\
		& \leq\left\{ (2C_1+C_4)T+C_3 \right\}\sum_{k=0}^{m}\int_{\Omega_k}\left[\beta^k \mid u^k \mid^2 +\alpha^k(p^k)^2 \right]dx + T\int_{0}^{T}\int_{S_0}\alpha p^2 						\frac{\partial h}{\partial\eta}
	\end{aligned}
\end{equation}
thus,
\begin{equation}{\label{Ultima_Estim}}
	\begin{aligned}
		& \frac{T}{2}\left[ 1-\delta_0(1-\mu(\Omega))\right]\sum_{k=0}^{m}\int_{\Omega_k} 				  \left[\beta^k \mid u^k \mid^2 + \alpha^k(p^k)^2 \right]dx -2C_5 
		  \sum_{k=0}^{m}\int_{\Omega_k}\left[\beta^k \mid u^k \mid^2 +\alpha^k(p^k)^2 \right]dx 		   \leq \\
		& 2 \int_{0}^{T}\int_{S_0}\alpha p^2\frac{\partial h}{\partial\eta}dS_0 dt
	\end{aligned}
\end{equation}
where $C_5=C_3+(2C_1+C_4)T$ and considered  $T>\max{1,(2C_1+C_4)}$. 
We have proved the 
\begin{theorem}{\label{Desig_Obser_Sist_1}}
Taking $\Phi$ as in (\ref{Funcao_Phi}), the geometry properties (\ref{Hipo_Omega}), and the hypothesis of monotony of coefficients (\ref{Mono_coef}) and the hypothesis (\ref{Hipo_U0}); these were made for the initial condition. Then, $\exists C_5>0$, with independence of $t,u,u_0,p_0$, such that 
\[
	\frac{T}{2}\left[ 1-\delta_0(1-\mu(\Omega))-2C_5\right]\sum_{k=0}^{m}\int_{\Omega_k} 			\left[\beta^k \mid u^k \mid^2 + \alpha^k(p^k)^2 \right]dx \leq 2 \int_{0}^{T}\int_{S_0}			\alpha p^2\frac{\partial h}{\partial\eta}dS_0 dt.
\]
\end{theorem}
The same manner, we obtain the inequality of observability for the system (\ref{Par_Ecu_Ondas_ST2})-(\ref{Par_Ecu_Ondas_Multilayer_2}) with their interface conditions and the monotonicity of the coefficients, given by:
\begin{align}{\label{Mono_2}}
\gamma^{k-1} &\leq  \gamma^{k}\\\nonumber
\tau^{k-1} &\leq  \tau^{k}
\end{align}
\begin{theorem}{\label{Desig_Obser_Sist_2}}
Assuming $\Phi$ as in (\ref{Funcao_Phi}), the monotonicity for the coefficients (\ref{Mono_2}) and the hypothesis of the theorem  \ref{Desig_Obser_Sist_1} with $h(x)=\frac{1}{2}\mid x-x_0\mid^2 + \delta_0\Phi(x)$ and  $(v_0,q_0)\in V_2\cap D(A_2)$, $v^{k}_{0}=\nabla m^k$, with $m^k\in H^2(\Omega_k)$, $m=0$ in  $S_1$. Then, there is a constant $C_6>0$, with independence of $t,v_0,q_0$ such that 
\[
\begin{aligned}
	& T\left[ 1-\delta_0(1-\mu(\Omega))\right]\sum_{k=0}^{m}\int_{\Omega_k}
		\left[\tau^k \mid v^k \mid^2 + \gamma^k(q^k)^2 \right]dx - 2C_6 \sum_{k=0}^{m}\int_{\Omega_k}\left[\tau^k \mid v^k \mid^2 + \gamma^k(q^k)^2 \right]dx\\
	&\leq  2 \int_{0}^{T}\int_{S_0}\tau\mid v.\eta \mid^2 \frac{\partial h}{\partial\eta}dS_0 dt.
\end{aligned}
\]
\end{theorem}
\begin{proof}
The proof was obtained using the results in \cite{Perla} and the estimations that were made  in the proof of the theorem \ref{Desig_Obser_Sist_1}
\end{proof}
Assuming the hypothesis of the theorems \ref{Desig_Obser_Sist_1} and  \ref{Desig_Obser_Sist_2}, we obtain the inequalities of the observability:
\begin{equation}{\label{Casi_Est_Final}}
\begin{aligned}
& (T-T_0)\sum_{k=0}^{m}\int_{\Omega_k}\left[ \beta^k \mid u^k \mid^2 + \alpha^k(p^k)^2 + \tau^k \mid v^k \mid^2 + \gamma^k(q^k)^2 \right]dx \\
&\leq C_7\int_{0}^{T}\int_{S_0}\left[ \alpha p^2+ \tau\mid v.\eta \mid^2 \right]\frac{\partial h}{\partial\eta}dS_0 dt.
\end{aligned}
\end{equation}
For any $T\geq T_0=\max\left\{ 1,\frac{C_5+C_6}{1-\delta_0(1-\mu(\Omega))}  \right\}$.
That was made in \cite{Perla} and  \cite{Kapitonov}, (\ref{Casi_Est_Final}). This is an inequality of observability, moreover, it is not convenient to use the H.U.M technique. Starting of  (\ref{Casi_Est_Final}), we obtain an appropriate inequality.
\begin{theorem}{\label{Desigualdade_Final}}
Assuming the hypothesis of the theorem  \ref{Desig_Obser_Sist_1},   \ref{Desig_Obser_Sist_2}. Moreover, we suppose  that :
\begin{equation}{\label{Aux_Theo_Final}}
	\begin{aligned}
	\alpha^k\beta^k &= \gamma^k\tau^k\\
	\beta^{k-1}\tau^k	&= \beta^k\tau^{k-1}.
	\end{aligned}
\end{equation}
Then, There is a positive constant $C>0$ such that
\[
	\begin{aligned}
		&(T-T_0)\sum_{k=0}^{m}\int_{\Omega_k}\left[ \beta^k \mid u^k \mid^2 + \alpha^k(p^k)^2 + 		\tau^k \mid v^k \mid^2 + \gamma^k(q^k)^2 \right]dx \\
		& \leq C\int_{0}^{T}\int_{S_0}\left[ \alpha p- \tau( v.\eta) \right]^2\frac{\partial h}{\partial\eta}dS_0 dt.
	\end{aligned}
\]
$\forall T>\max\left\{ 1, \frac{C_5+C_6}{1-\delta_0(1-\mu(\Omega))} \right\}$.
\end{theorem}
\begin{proof}
Using the theorems (\ref{Par_Ecu_Ondas_Multilayer_2}) and  (\ref{Aux_Theo_Final}), we have  that
\begin{equation}{\label{Aux_Theo_Final_2}}
	\begin{aligned}
		\frac{d}{dt}\int_{\Omega}(\tau u.v+\alpha p q)dx &= \frac{d}{dt}\sum_{k=0}^{m}					 \int_{\Omega_k}(\tau^{k} u^{k}.v^{k}+\alpha^{k} p^{k} q^{k})dx 
		= \sum_{k=0}^{m}\int_{\Omega_k}\left\{ \tau^k\frac{\partial u^k}{\partial t}.v^k + 			     \tau^k u^k.\frac{\partial v^k}{\partial t} +\right. \\
		& \left. \alpha^k \frac{\partial p^k}{\partial t}q^k + \alpha^k p^k \frac{\partial q^k}				   {\partial t} \right\} \\
		&= \sum_{k=0}^{m}\int_{\Omega_k}\left\{ -\alpha^k\tau^k\nabla p^k.v^k - \tau^k\gamma^k			   u^k\nabla q^k -\alpha^k\beta^k\mbox{div}(u^k) q^k - \alpha^k\tau^k p^k \mbox{div}			  (v^k)\right\}\\
		&= -\sum_{k=0}^{m}\int_{\partial\Omega_k}\alpha^k \tau^k p^k(v^k.\eta)+ \sum_{k=0}^{m}			    \int_{\Omega_k}\left\{ \alpha^k\tau^k p^k\mbox{div}(v^k)- \tau^k\gamma^k	u^k					\nabla q^k -\right. \\
   &\left. \alpha^k\beta^k\mbox{div}(u^k)q^k -\alpha^k\tau^k p^k \mbox{div}(v^k)\right\}dx \\
	&= -\sum_{k=0}^{m}\int_{\partial\Omega_k}\alpha^k \tau^k p^k(v^k.\eta)-\sum_{k=0}^{m}			\int_{\partial\Omega_k}\tau^k\gamma^k q^k(u^k.\eta)+\\
	&  \sum_{k=0}^{m}\int_{\partial\Omega_k}\left\{  \gamma^k\tau^k\mbox{div}(u^k)q^k - 			\alpha^k\beta^k\mbox{div}(u^k)q^k  \right\}\\
	&= -\sum_{k=0}^{m}\int_{\partial\Omega_k}\alpha^k \tau^k p^k(v^k.\eta)-\sum_{k=0}^{m}			\int_{\partial\Omega_k}\tau^k\gamma^k q^k(u^k.\eta)\\
	&= -\int_{S_1}\alpha\tau p(v.\eta)-\sum_{k=0}^{m}\int_{\Gamma_k}\left[ \alpha^{k-1} 			   \tau^{k-1}p^{k-1}(v^{k-1}.\eta) - \alpha^k \tau^k p^k(v^k.\eta)  \right] - \\
	& \int_{S_0}\alpha\tau p(v.\eta) - \int_{S_1}\tau\gamma q(u.\eta) - \sum_{k=0}^{m}					\int_{\Gamma_k}\left[ \tau^{k-1}\gamma^{k-1}q^{k-1}(u^{k-1}.\eta) - \tau^{k}					\gamma^{k}q^{k}(u^{k}.\eta) \right] - \\
	&\int_{S_0}\tau\gamma q(u.\eta)
	\end{aligned}
\end{equation}
and, using the boundary condition and (\ref{Aux_Theo_Final}) in the interfaces, we have that 
\begin{equation}{\label{Aux_TF}}
\frac{d}{dt}\sum_{k=0}^{m}\int_{\Omega_k}(\tau^{k} u^{k}.v^{k}+\alpha^{k} p^{k} q^{k})dx = -\int_{S_0}\alpha\tau p(v.\eta)dS_0
\end{equation}
This is clearly that the first, second, fourth and sixth term of (\ref{Aux_Theo_Final_2}) are null only substituting directly the interface conditions. The fifth term is made using (\ref{Aux_Theo_Final})
\begin{equation}{\label{Aux_TF_1}}
	\begin{aligned}
		\sum_{k=0}^{m}\int_{\Gamma_k}\left[ \tau^{k-1}\gamma^{k-1}q^{k-1}(u^{k-1}.\eta) - 				\tau^{k}\gamma^{k}q^{k}(u^{k}.\eta) \right]	&= 
		\sum_{k=0}^{m}\int_{\Gamma_k}\left[ \frac{\tau^{k-1}}{\beta^{k-1}}\beta^{k-1}					\gamma^{k-1}q^{k-1}(u^{k-1}.\eta) -\right.\\
		& \left.\frac{\tau^{k}}{\beta^{k}}\beta^{k}\gamma^{k}q^{k}(u^{k}.\eta) \right]\\
		& =	\sum_{k=0}^{m}\int_{\Gamma_k}\left( \frac{\tau^{k-1}}{\beta^{k-1}} - 						    \frac{\tau^{k}}{\beta^{k}} \right)\gamma^k q^k\beta^k(u^k.\eta)\\
		& = \sum_{k=0}^{m}\int_{\Gamma_k}\frac{\gamma^{k}}{\beta^{k-1}}(\beta^{k-1}\tau^{k-1}-			  \beta^{k-1}\tau^k)q^{k}(u^k.\eta) \\
		&=0
	\end{aligned}
\end{equation}
and ,
\begin{equation}{\label{Aux_Front}}
	\int_{0}^{T}\int_{S_0}\left[   \alpha p -\tau(v.\eta)  \right]^2 = \int_{0}^{T}\int_{S_0}		 \alpha^2 p^2 + \int_{0}^{T}\int_{S_0}\tau^2(v.\eta)^2 - 2\tau\alpha\int_{0}^{T}\int_{S_0} 		p(v.\eta)
\end{equation}
Integrating (\ref{Aux_TF}) in $(0,T)$ and substituting in (\ref{Aux_Front}), we have that
\[
\int_{0}^{T}\int_{S_0}\left[ \alpha p +\tau(v.\eta) \right]^2 = \int_{0}^{T}\int_{S_0}\alpha^2 p^2 + \int_{0}^{T}\int_{S_0}\tau^2 (v.\eta)^2 + \sum_{k=0}^{m}\int_{\Omega_k}(\tau^k u^k.v^k + \alpha^k p^k q^k)\mid_{0}^{T}
\]
Now, we have that 
\begin{equation}{\label{Aux_Fin_F}}
	\int_{0}^{T}\int_{S_0}\left( \alpha^2 p^2 +\tau^2 (v.\eta)^2  \right) = \int_{0}^{T}			\int_{S_0}\left[  \alpha p - \tau(v.\eta)  \right]^2 - \sum_{k=0}^{m}\int_{\Omega_k}(\tau^k 	u^k.v^k + \alpha^k p^k q^k)\mid_{0}^{T}
\end{equation}
using (\ref{Aux_Fin_F}) to estimate the term of right-hand side of the inequality (\ref{Casi_Est_Final}) 
\begin{equation}{\label{Aux_Fin_F_F}}
	\begin{aligned}
		\int_{0}^{T}\int_{S_0}\left[ \alpha p^2 + \tau(v.\eta)^2 \right] &\leq \max\{ 					 (\alpha^{-1}),(\tau^{-1}) \}\int_{0}^{T}\int_{S_0}\left[   \alpha^2 p^2 +\tau^2 (v.			 \eta)^2  \right] \\
		& \leq C_8\int_{0}^{T}\int_{S_0}\left[  \alpha p -\tau(v.\eta)  \right]^2 -						  C_8\sum_{k=0}^{m}\int_{\Omega_k}(\tau^k u^k.v^k + \alpha^k p^k q^k)\mid_{0}^{T}
	\end{aligned}		
\end{equation}
Substituting in (\ref{Casi_Est_Final}), we have that
\begin{equation}{\label{Casi_Est_Final_3}}
\begin{aligned}
& (T-T_0)\sum_{k=0}^{m}\int_{\Omega_k}\left[ \beta^k \mid u^k \mid^2 + \alpha^k(p^k)^2 + \tau^k \mid v^k \mid^2 + \gamma^k(q^k)^2 \right]dx + C_8\sum_{k=0}^{m}\int_{\Omega_k}(\tau^k u^k.v^k + \alpha^k p^k q^k)\mid_{0}^{T} \\
& \leq C_7C_8\int_{0}^{T}\int_{S_0}\left[  \alpha p -\tau(v.\eta)  \right]^2
\end{aligned}
\end{equation}
and the proof of the theorem follows the inequalities:
\begin{equation}{\label{Aux_Casi_F}}
\begin{aligned}
\sum_{k=0}^{m}\int_{\Omega_k}(\tau^k u^k.v^k + \alpha^k p^k q^k)\mid_{0}^{T} \leq C_9\left[  \beta^k\mid u^k \mid^2 + \alpha^k(p^k)^2 +\tau^k\mid v^k \mid^2 +\gamma^k(q^k)^2  \right]dx
\end{aligned}
\end{equation}
where $C_9=\max\left\{ \tau^{k}(\beta^k)^{-1}, \alpha^k(\tau^k)^{-1}   \right\}$
\end{proof}
As a corollary , of uniqueness of the theorem  \ref{Desigualdade_Final}, we have that:
\begin{corollary}{\label{Produto_Interno_D}}
with the hypothesis of the theorem \ref{Desigualdade_Final}, given $\{u,p\}$ and  $\{v,q\}$ solutions of the problem (\ref{Par_Ecu_Ondas_Multilayer_1}) and  (\ref{Par_Ecu_Ondas_ST2}), respectively. Making that  
\[\alpha p(x,t)=\tau v(x,t).\eta\quad\quad\forall (x,t)\in\Gamma_0 \times (0,T)\]
\end{corollary}
then, in $T>T_0$, $u=v=0$ and  $p=q=0$ for all $(x,t)\in\Omega\times (0,T)$

\section{Exact controllability}

As a consequence of the corollary \ref{Produto_Interno_D} , we have that: for  $T>T_0$, the expression  
\begin{equation}{\label{Formula_PI}}
\left[  \int_{0}^{T}\int_{\Gamma_0}\left[  \alpha p-\tau (v.\eta)  \right]  \right]^{1/2}
\end{equation}
define a norm in a space of initial data $(u_0,p_0)$ and  $(v_0,q_0)$ the problems (\ref{Par_Ecu_Ondas_Multilayer_1}) and  (\ref{Par_Ecu_Ondas_ST2}). We denote by $Y$ the Hilbert space defined as closure of $V_1\cap D(A_1)\times V_1\cap D(A_1)$ in  $X=X_1\times X_2$ with the norm (\ref{Formula_PI}). The obtained number in (\ref{Formula_PI}) is denoted by  $\parallel(u_0,p_0,v_0,v_0,q_0)\parallel_{Y}$.
Now,  $Y\subset X$ and  
\begin{equation}{\label{Compara_Normas}}
\begin{aligned}
\parallel (u_0,p_0,v_0,q_0) \parallel^{2}_{X} &= \parallel (u_0,p_0) \parallel^{2}_{X_1} + \parallel (v_0,q_0) \parallel^{2}_{X_2} \\
& \leq C \parallel (u_0,p_0,v_0,q_0) \parallel^{2}_{Y}
\end{aligned}
\end{equation}
for some positive constant $C$.
The dual space  of $Y$  respect to  $X$ is denoted by $Y^{'}$. In  $\Omega\times (0,T)$ consider the systems (\ref{Par_Ecu_Ondas_Multilayer_1}) and (\ref{Par_Ecu_Ondas_Multilayer_2}) with initial condition $(u_0,p_0,v_0,q_0)\in Y^{'}$.  Using the transportation method, the solution to the problems (\ref{Par_Ecu_Ondas_Multilayer_1}) and (\ref{Par_Ecu_Ondas_Multilayer_2}) with no homogeneous contour conditions.
\begin{definition}
Given  $(u(x,t),p(x,t),v(x,t),q(x,t))\in C(O,T; Y^{'})$, is a solution of  (\ref{Par_Ecu_Ondas_Multilayer_1}) and  (\ref{Par_Ecu_Ondas_Multilayer_2}), if
\begin{equation}{\label{Dualidad_X}}
\left< (u,p,v,q),(\tilde{u},\tilde{p},\tilde{v},\tilde{q})  \right>_{X} = \left< (u_0,p_0,v_0,q_0),(\tilde{u}_0,\tilde{p}_0,\tilde{v}_0,\tilde{q}_0)  \right>_{X} - \int_{0}^{T}\int_{\Gamma_0}\left[ \alpha\beta Q \tilde{p} + \gamma\tau P (\tilde{n}.\eta)  \right]d\Gamma_0 ds
\end{equation}
for all  $(\tilde{u}_0,\tilde{p}_0,\tilde{v}_0,\tilde{q}_0)\in Y$ and  $0<t<T$. Here $(\tilde{u}_0,\tilde{p}_0)$ and  $(\tilde{v}_0,\tilde{q}_0)$ are the solutions of (\ref{Par_Ecu_Ondas_Multilayer_1}) and  (\ref{Par_Ecu_Ondas_Multilayer_2}), respectively, for the functions $P, Q\in C(0,T;L^{2}(\Gamma_0))$.
In (\ref{Dualidad_X}) is given by 
\[
\left< (u,p,v,q),(\tilde{u},\tilde{p},\tilde{v},\tilde{q})  \right>_{X}= \left< (u,p),(\tilde{u},\tilde{p}\right>_{X_1} + \left< (v,q),(\tilde{v},\tilde{q})  \right>_{X_2}
\]
\end{definition}
\begin{definition}{\label{Cond_0}}
A solution of (\ref{Par_Ecu_Ondas_Multilayer_1}) and (\ref{Par_Ecu_Ondas_Multilayer_2}); This is null in the time $t=T$. The function  $(u(x,t),p(x,t),v(x,t),q(x,t))(\ref{Par_Ecu_Ondas_Multilayer_1})$ and $(\ref{Par_Ecu_Ondas_Multilayer_2})in C(0,T;Y^{'})$ such that  
\begin{equation}{\label{Eq_C0}}
\left< (u,p,v,q),(\tilde{u},\tilde{p},\tilde{v},\tilde{q})  \right>_{X}= \int_{0}^{T}\int_{\Gamma_0}\left[ \alpha\beta Q \tilde{p} + \gamma\tau P (\tilde{n}.\eta)  \right]d\Gamma_0 ds
\end{equation}
for all $(\tilde{u},\tilde{p},\tilde{v},\tilde{q})\in Y$ and $0<t<T$
\end{definition}
Given the lineal and reversible systems (\ref{Par_Ecu_Ondas_Multilayer_1}) and (\ref{Par_Ecu_Ondas_Multilayer_2}) in the time; it is clearly to solve the problem of exact controllability, it is sufficient to prove that, for all initial condition in $Y^{'}$, and their solutions, it can be take in the equilibrium of time $T$.\\
Given $G_1=(w_0,k_0)$ and $G_2=(m_0,l_0)$ arbitrary elements of $Y$. We denote by 
\[
\begin{aligned}
&(w(x,t),k(x,t))=U_1(t)(w_0,k_0)\\
&(m(x,t),l(x,t))=U_2(t)(m_0,l_0)
\end{aligned}
\]
consider the following functions
\begin{equation}{\label{Funcoes_Bordo}}
\begin{aligned}
&Q=\beta\left(  \alpha k(x,t)-\tau m(x,t).\eta \right)\\
&P=-\frac{\beta}{\gamma}Q
\end{aligned}
\end{equation}
and given $(u,p)$ and  $(v,q)$ the solution of (\ref{Par_Ecu_Ondas_Multilayer_1}) and  (\ref{Par_Ecu_Ondas_Multilayer_2}), these are null in the instant $T$, $(T>T_0)$ and the contour conditions (\ref{Funcoes_Bordo}).\\
Considering the map 
\[
\begin{aligned}
\wedge:\quad &Y\longrightarrow Y^{'}\\
&(G_1,G_2)\longrightarrow \wedge(G_1,G_2)=(u,p,v,q)\mid_{t=0}
\end{aligned}
\]
Using (\ref{Eq_C0}) in $t=0$ and substituting $P$ and  $Q$ given by (\ref{Funcoes_Bordo}), we have
\begin{equation}{\label{Isomorfismo_F}}
\begin{aligned}
\left<\wedge(G_1,G_2),(\tilde{u}_0,\tilde{p}_0,\tilde{v}_0,\tilde{q}_0)  \right>_{X} &= \int_{0}^{T}\int_{\Gamma_0}\left[ \alpha \beta Q \tilde{p} + \gamma\tau P \tilde{v}.\eta\right]d\Gamma_0 ds \\
&= \int_{0}^{T}\int_{\Gamma_0}(\gamma k-\tau m.\eta)(\alpha\tilde{p}-\tau \tilde{v}.\eta)d\Gamma_0 ds \\
&= \left<(G_1,G_2),(\tilde{u}_0,\tilde{p}_0,\tilde{v}_0,\tilde{q}_0)  \right>_{Y}
\end{aligned}
\end{equation}
of (\ref{Isomorfismo_F}), we can conclude that  $\wedge$ is an isomorphism of  $Y$ in $Y^{'}$. Putting 
\begin{equation}{\label{Funcoes_Bordo_Final}}
\begin{aligned}
(G_1,G_2) &= \wedge^{-1}((u_0,p_0),(v_0,q_0)) \\
Q &=\beta^{-1}\left(  \alpha k(x,t)-\tau m(x,t).\eta \right)\\
P &=-\frac{\beta}{\gamma}Q
\end{aligned}
\end{equation}
Using (\ref{Dualidad_X}) with  $t=T>T_0$, we have that	 
\[
\begin{aligned}
&\left< (u(x,T),p(x,T),v(x,T),q(x,T)),(U_1(t)(\tilde{u}_0,\tilde{p}_0),U_2(t)(\tilde{v},\tilde{q}))  \right>_{X} = \\
&= \left<\wedge(G_1,G_2),(\tilde{u}_0,\tilde{p}_0,\tilde{v}_0,\tilde{q}_0)  \right>_{X} - \left<(G_1,G_2),(\tilde{u}_0,\tilde{p}_0,\tilde{v}_0,\tilde{q}_0)\right>_{Y} 
\end{aligned}
\]
for all $(\tilde{u}_0,\tilde{p}_0,\tilde{v}_0,\tilde{q}_0)\in Y$. Using (\ref{Isomorfismo_F}), we have that $(u(x,T),p(x,T),v(x,T),q(x,T))$ is a functional null on $Y$. In conclusion we prove the following theorem
\begin{theorem}{\label{TF}}
Assuming the hypothesis of theorem \ref{Desigualdade_Final}. If given $T>T_0$  and initial condition  $(u_0,p_0,v_0,q_0)\in Y^{'}$ the problem  (\ref{Par_Ecu_Ondas_ST1}), (\ref{Par_Ecu_Ondas_ST2}), (\ref{Cond_Multilayer_1}), (\ref{Cond_Multilayer_1}). Then, there is exist an control $Q(x,t)\in C(0,T;L^{2}(\Gamma_0))$ such that the corresponding solution $(u,p,v,q)$ with the boundary condition 
\begin{equation}
	\left\{\begin{aligned}
		& u.\eta  = Q, \quad \mbox{in}\quad S_0\times(0,T)\\
		& p = 0, \quad \mbox{in}\quad S_1\times(0,T)\\
		& u(x,0) = u_{0}(x), p(x,0) =p_{0}(x)
 	\end{aligned}\right.	
\end{equation}
and
\begin{equation}
	\left\{\begin{aligned}
		& q  = P, \quad \mbox{in}\quad S_0\times(0,T)\\
		& q = 0, \quad \mbox{in}\quad S_1\times(0,T)\\
		& v(x,0) =v_{0}(x), q(x,0) =q_{0}(x)
 		\end{aligned}\right.
\end{equation}
with $P=-\beta\gamma^{-1}Q$, satisfy , for $x\in \Omega$
\begin{equation}
	\left\{\begin{aligned}
		& u(x,T)  = 0\\
		& p(x,T)  = 0\\
		& v(x,T)  = 0\\
		& q(x,T)  = 0
 		\end{aligned}\right.
\end{equation}
\end{theorem}

\section{Acknowledgment}
%El autor agradece a los profesores G. Perla Menzala y Boris Kapitonov por sus sugerencias y correcciones al presente articulo.
We thank to professors G. Perla Mezala and Boris Kapitonov for the suggestions in this article.


\begin{thebibliography}{a}
\bibitem{Adams}R. A. Adams, Sobolev Spaces, Academic Press, New York, 1975.

\bibitem{Agmon}S. Agmon,  A. Douglis, L. Nirenberg, Estimates near the boundary for the solutions of elliptic partial equations satisfying general boundary conditions II. Comm. Pure Appl. Math., 17(1), 35?92 (1964)

\bibitem{Nirenberg}S. Agmon, A. Douglis, and L. Nirenberg, Estimates near the boundary for the solutions of elliptic partial equations satisfying general boundary conditions II, Comm. Pure Appl. Math. 17(1964, 35-92)

\bibitem{Bardos}C. Bardos, G. Lebeau, and J. Rauch, Sharp sufficient conditions for the observation, control and stabilization of wave from the boundary, SIAM J. Control Optim. 30(1992), 1024-1065.

\bibitem{Coron} J. M. Coron, Control and nonlinearity, Mathematical Survey and Monographs 2009.

\bibitem{Komornik}V. Komornik, Exact Controllability and Stabilization. The multiplier method, RAM:Research in Applied Mathematics. Masson, Paris; Jhon Wiley and Sons, Chichester, 1994.

\bibitem{Lagnese_3}J. E. Lagnese, Boundary controllability in problems of transmission for a class of second order hyperbolic systems, ESAIM Control Optim. Calc Var., 2 (1997), 343-357

\bibitem{Kapitonov}  B.V. Kapitonov, and G. Perla , Simultaneous exact controllability for a pair of systems of evolution of sound in a compressible fluid, Journal of Mathematics and System Science 3 (2013) 655-658

\bibitem{Kapitonov_1}B.V. Kapitonov, and G.Perla, Boundary Stabilization and a problem of transmission for a system of propagation of sound, Funkcialaj Ekvacioj, 49 (2006), 107--132.

\bibitem{Perla} B.V. Kapitonov, and G. Perla, Simultaneous exact controllability: an elastodynamic system and Maxwell's equations,     Adv. Differential Equations Volume 16, Number 5/6 (2011), 551-571.

\bibitem{Landau}L.D. Landau and E.M. Lifshitz, Fluid Mechanics Volume 6 of Course of Theorical Physics Pergamon Press, Oxford, New York, Paris (1960)

\bibitem{Leis}  R. Leis, Initial boundary value problems in Mathematical Physics, John Wiley, New York, 1986

\bibitem{Lions}J. L. Lions, Extact controllability, Stabilization and perturbations for distributed system, SIAM Review 30(1988), 1-68.

\bibitem{Miara_Perla}B. Miara and G. Perla, Exact Controllability Of Naghdi Shells, Comptes Rendus de l'Academia des Sciences, Paris, Serie Mathematique - Vol. 348 - Issue 5-6 - 2010, 341-346

\bibitem{Naghdi}P. M. Naghdi, Foundations of elastic shell theory, 1963. Progress in Solid Mechanics, Vol IV, pp. 1-90, North-Holland, Amsterdam.

\bibitem{Naghdi_1}P. M. Naghdi, Theory of shells and Plates, Handbuch der physik, Vol VI. Springer-Verlag, Berlin, 1972.

\bibitem{Pazy}A. Pazy, Semigroup of Linear Operator and Application to Partial Differential equations, Springer-Verlag, New York, Berlin, 1983.

\bibitem{Petersen}P. Petersen, Riemannian Geometry, Second edition. Graduate Texts in Mathematics, Springer, New York, 2006.

\bibitem{Russell}D. L. Russell, A unified boundary controllability theory for hyperbolic and parabólic partial differential equations, Studies in Appl. Math. 52(1973),189-211.

\bibitem{Russell_2}D. L. Russell, Controllability and Stabilizability theory for linear partial differential equations: recent progress an open questions, SIAM Review 20(4)(1978), 639-739.

\bibitem{Zuazua_1}E. Zuazua, Exponential decay for the semilinear wave equation with locally distributed damping, Comm. Partial Differential Equations 15(1990), 205-235.

\bibitem{Zuazua_2}E. Zuazua, Expoential decay for the semilinear wave equation with locally distributed damping in unbound domains, J. Math. Pures Appl 70(9) 70 (1991), 513-529.


\end{thebibliography}
\end{document}